\documentclass{article}
\usepackage{filecontents}
\usepackage{authblk} 
\usepackage{fancyhdr} 
\usepackage{amsmath,amsthm,amssymb,xcolor,comment,url}
\usepackage{enumitem,etoolbox}
\usepackage{hyperref}
\usepackage[textwidth=17cm, textheight=22cm, bottom=30mm]{geometry}

\newcommand{\zerodisplayskips}{%
  \setlength{\abovedisplayskip}{0.05in}%
  \setlength{\belowdisplayskip}{0.05in}%
  \setlength{\abovedisplayshortskip}{0in}%
  \setlength{\belowdisplayshortskip}{0.05in}}
\appto{\small}{\zerodisplayskips}
\appto{\footnotesize}{\zerodisplayskips}

\usepackage[compact]{titlesec}
\titlespacing{\section}{0pt}{0.15in}{0.1in}
\titlespacing{\subsection}{0pt}{0.05in}{0.05in}

\newtheorem{theorem}{Theorem}
\newtheorem{lemma}[theorem]{Lemma}
\newtheorem{corollary}[theorem]{Corollary}
\newtheorem{proposition}[theorem]{Proposition}
\newtheorem{definition}{Definition}
\newtheorem*{mdprob}{Maximal determinant problem}

\title{A Survey of the Hadamard Maximal Determinant Problem}

\author{Patrick Browne\thanks{Technological University of the Shannon: Midlands Midwest, Ireland} \quad Ronan Egan\thanks{Dublin City University, Ireland} \quad Fintan Hegarty\thanks{Mathematical Sciences Publishers, Berkeley, CA, USA} \quad Padraig \'O Cath\'ain\thanks{Worcester Polytechnic Institute, MA, USA}}
\date{}

\AtBeginDocument{\let\mod\bmod}

\begin{document}
\maketitle
\begin{abstract}
In a celebrated paper of 1893, Hadamard
proved the maximal determinant theorem, which establishes an upper bound on the determinant of a matrix with complex entries of norm at most $1$. His paper concludes with the suggestion that mathematicians study the maximum value of the determinant of an $n \times n$ matrix with entries in $\{ \pm 1\}$. This is the Hadamard maximal determinant problem.

This survey provides complete proofs of the major results obtained thus far. We focus equally on upper bounds for the determinant (achieved largely via the study of the Gram matrices), and constructive lower bounds (achieved largely via quadratic residues in finite fields and concepts from design theory). To provide an impression of the historical development of the subject, we have attempted to modernise many of the original proofs, while maintaining the underlying ideas. Thus some of the proofs have the flavour of determinant theory, and some appear in print in English for the first time.

We survey constructions of matrices in order $n \equiv 3 \mod 4$, giving asymptotic analysis which has not previously appeared in the literature. We prove that there exists an infinite family of matrices achieving at least $0.48$ of the maximal determinant bound. Previously the best known constant for a result of this type was $0.34$.

\end{abstract}
\bigskip


The story, of course, does not begin with Hadamard.
Thomson conjectured in 1885 a bound on the determinant of a matrix in terms of the norms of its rows; this was established shortly afterward by Muir.
In his \textit{R\'{e}solution d'une question relative aux d\'{e}terminants} of 1893,
Hadamard gives (i) a proof of the so-called \emph{Hadamard determinant bound} (which is essentially the Muir--Thomson bound), (ii) an explicit statement of the \textit{{m}aximal determinant
problem} (for $\mathbb{R}$), and
(iii) solutions to this problem at orders $2^{t}$, $12$ and $20$. Nevertheless, in Section \ref{secHad} we follow Hadamard's exposition
(his paper being as readable today as in 1893), before tracing a little of the history of the determinant bound before and after
Hadamard. Again following the original exposition, we give the proof of Fischer's {i}nequality using compound matrices, which generalises Hadamard. In Section \ref{SecMaxDet} we describe the (real) {m}aximal {d}eterminant problem, which is to construct $\{ \pm 1\}$ matrices
attaining the {d}eterminant {b}ound, and describe some results from the theory of Hadamard matrices.

In quick succession in the 1960s, Ehlich and Wojtas produced sharper bounds than Hadamard's for $\{ \pm 1\}$ matrices of orders
$n \equiv 1, 2, 3 \mod 4$. Their bounds are presented in Section \ref{SecEW}. Each of the three cases has its own peculiarities, discussed in
turn in Sections \ref{Sec1mod4}, \ref{Sec2mod4} and \ref{Sec3mod4} respectively. In each case, we survey the known constructions
which achieve a determinant within a constant factor of the best known bound, and comment on computational and theoretical work at small orders. In Section \ref{newbounds} we analyse the known theoretical constructions for matrices with large determinant at orders $n \equiv 3 \mod 4$. We generalise a construction of Neubauer and Radcliffe, allowing us to prove that there exists an infinite family of matrices exceeding $0.48$ of the Ehlich bound.

In writing this paper, the authors made the conscious decision to present the main results for the maximal determinant problem
with some historical context. Thus our presentation is approximately chronological, and we attempt to follow the techniques of the original
authors. These choices result in some heterogeneity of style: Hadamard worked with Hermitian matrices while Fischer worked with real symmetric matrices, for example, and we have not attempted to reconcile these accounts. We perceive two underlying themes which run through many
proofs in this area.
\begin{enumerate}
	\item The Gram matrix of a real-valued matrix is symmetric positive definite. All its eigenvalues are real and positive. Most of the
	determinant bounds that we present use linearity of the determinant in the rows of a matrix to express the determinant as the sum of a positive and a negative term. The positive term becomes an upper bound on the determinant, and minimising the negative term saturates the corresponding bound.
	Slightly intricate induction hypotheses appear to be a necessary feature of these proofs. Theorem \ref{Had} is the prototype of this result,
 and Theorems \ref{Wojtas} and \ref{2mod4} follow the same pattern, which reaches its most developed form in the results of Section \ref{Sec3mod4}.
	\item Going hand-in-hand with non-constructive upper bounds are constructive lower bounds. Outside the large literature on Hadamard matrices, there are only a few construction techniques. Direct constructions use combinatorial designs obtained from finite fields (specifically affine planes and quadratic residue designs): these techniques are introduced in Section \ref{SecPaley} and developed more extensively in Section \ref{RonanMagic}. Tensor products and block matrices assembled from these basic matrices produce further results: Proposition~\ref{FK} and Corollary \ref{2mod4} are easy examples of this method; Section \ref{RonanMagic} again supplies the most detailed applications.
	\end{enumerate}

\section{The Hadamard determinant bound}
\label{secHad}

A curiosity of Hadamard's paper to the eye of the modern reader is the absence of concepts from linear algebra.
For Hadamard, a matrix is nothing but an array from which the determinant (considered a homogeneous polynomial function of
degree $n$ in $n^{2}$ variables) is computed. Our proof follows Hadamard's, with notation modernised and what Hadamard
refers to
as an \textit{identit\'{e} bien connu} presented explicitly.

In this paper, matrices are square unless stated otherwise. We use $I_{n}$ and $J_{n}$ to denote the $n \times n$ identity and all-ones matrices respectively, and drop the subscript when the order is clear from context. Recall that a matrix is \textit{Hermitian} if $G^{\ast} = G$, and \textit{positive definite} if its eigenvalues are positive real numbers.
The \emph{Gram matrix} of $M$ is the matrix $MM^{\ast}$, which has as entries the inner products of rows of $M$.
A positive definite Hermitian matrix $G$ is a Gram matrix: via the square root of a positive matrix, it can be shown that
there exists a matrix~$X$ such that $XX^{\ast} = G$. Conversely, the Gram matrix of a set of linearly independent vectors is
Hermitian positive definite. There is a well-developed theory for positive definite matrices;
see for example the monograph of Horn and Johnson \cite{HornJohnson}. We follow Hadamard in considering a
\textit{minor} of order $k$ to be the determinant of a $k\times k$ submatrix.

\begin{theorem}[Paragraphes 2--4, \cite{Hadamard1893}]\label{Had}
Let $M$ be an $n \times n$ matrix with entries from the complex unit disk. Then $|\det(M)| \leq n^{n/2}$.
\end{theorem}

\begin{proof}
Define $G = MM^{\ast}$, and recall that the $(i,j)$ entry of $G$, which we denote $g_{i,j}$, is
the inner product of rows $i$ and $j$ of $M$.
Since its diagonal entries are real and $g_{i,j} = g_{j,i}^{\ast}$, the matrix $G$ is \textit{Hermitian}. Furthermore, $\det(G) =
\det(M)\det(M^{\ast})$, being the product of a complex number and its conjugate, is real and non-negative.

For a subset $\mathcal{I}$ of $\{1, 2, \ldots, n\}$, denote by $G_{\mathcal{I}}$ the principal submatrix of $G$ with rows and columns
indexed by $\mathcal{I}$. We write $P_{\mathcal{I}}$ for $\det(G_{\mathcal{I}})$ and $N_{\mathcal{I}}$ for the determinant obtained upon setting the bottom-right entry of $G_{\mathcal{I}}$ to zero. If $\mathcal{I} = \{1, 2, \ldots, k\}$ then we write $P_{k}$ for $P_{\mathcal{I}}$ and $N_{k}$ for $N_{\mathcal{I}}$. Since the determinant is linear in the rows of the matrix,
\begin{equation}\label{QQ}
\arraycolsep1.5pt
\det
\begin{pmatrix}
g_{1,1} & \ldots & g_{1,k-1} & g_{1,k} \\
g_{2,1}  & \ldots &g_{2,k-1} & g_{2,k} \\
\vdots & \vdots & \vdots & \vdots \\
g_{k-1,1}  & \ldots &g_{k-1,k-1} & g_{k-1,k} \\
g_{k,1} &  \ldots &g_{k,k-1} & g_{k,k}
\end{pmatrix} = \det
\begin{pmatrix}
g_{1,1} & \ldots & g_{1,k-1} & g_{1,k} \\
g_{2,1}  & \ldots &g_{2,k-1} & g_{2,k} \\
\vdots & \vdots & \vdots & \vdots \\
g_{k-1,1}  & \ldots &g_{k-1,k-1} & g_{k-1,k} \\
0 &  \ldots & 0 & g_{k,k}
\end{pmatrix} + \det
\begin{pmatrix}
g_{1,1} & \ldots & g_{1,k-1} & g_{1,k} \\
g_{2,1}  & \ldots &g_{2,k-1} & g_{2,k} \\
\vdots & \vdots & \vdots & \vdots \\
g_{k-1,1}  & \ldots &g_{k-1,k-1} & g_{k-1,k} \\
g_{k,1} &  \ldots &g_{k,k-1} & 0
\end{pmatrix}\!.
\end{equation}

Equation \eqref{QQ} illustrates the Laplace expansion of the determinant of $G_{k}$. We gather all terms containing $g_{k,k}$ and see that
\begin{equation}\label{recursive}
P_{k} = g_{k,k}P_{k-1} + N_{k} .
\end{equation}

By induction on $|\mathcal{I}|$ we will establish that $N_{\mathcal{I}}$ is always non-positive.
For this we require a general determinantal identity.
Let $U_{1}$ and $U_{4}$ be invertible square matrices of size $k \times k$ and $(n-k) \times (n-k)$ respectively.
For any $U_{2}$ and $U_{3}$ such that the displayed matrix $U$ is invertible, set $V = U^{-1}$, and decompose into blocks as in $U$:
\[ U = \begin{pmatrix}
U_{1} & U_{2} \\ U_{3} & U_{4}
\end{pmatrix}\!,
\quad
V =
\begin{pmatrix}
V_{1} & V_{2} \\ V_{3} & V_{4}
\end{pmatrix}\!.
\]
Now, take determinants on both sides of the expression
\begin{equation}\label{SchurComplement}
\begin{pmatrix}
U_{1} & U_{2} \\ U_{3} & U_{4}
\end{pmatrix}
\begin{pmatrix}
V_{1} & 0 \\ V_{3} & I
\end{pmatrix}
= \begin{pmatrix}
I & U_{2} \\ 0 & U_{4} \end{pmatrix}
\end{equation}
to see that
\begin{equation}\label{detid}
\det(U) \det(V_{1}) = \det(U_{4}).
\end{equation}

We return to our inductive proof. Suppose that $\mathcal{I} = \{i,j\}$. Recalling that $g_{i,j} = g_{j,i}^{\ast}$ because $G$ is
Hermitian,
\[ N_{\mathcal{I}} = \det
\begin{pmatrix}
g_{i,i} & g_{i,j} \\ g_{j,i} & 0
\end{pmatrix} = - |g_{i,j}|^{2},\]
and the result is established for the base case $|\mathcal{I}| = 2$. Suppose now that the inductive hypothesis holds for all~$\mathcal{I}$
for which $|\mathcal{I}| \leq k-1$. For notational convenience, we will work with the set $\{1, 2, \ldots, k\}$, but $\mathcal{I}$ can be
taken to be arbitrary of size $k$.
Take $V$ to be the rightmost matrix displayed in Equation \eqref{QQ}, so that $\det(V) = N_{k}$. If $\det(N_{k}) = 0$ the induction hypothesis
holds, so suppose that $N_{k}$ is invertible. Let $V_{1}$ be $G_{k-2}$, which is the submatrix of
$G_{k}$ containing the first $k-2$ rows and columns.
 The entries of $U = V^{-1}$ are the $(k-1) \times (k-1)$ cofactors of $V$. So $\det(U) = N_{k}^{-1}$ and $\det(V_{1}) =
 P_{k-2}$.
We denote by $\gamma$ the (non-principal) minor obtained by deleting row $k-1$ and column $k$ of $V$. Then up to some $(-1)$ factors which
cancel in the determinant,
\[
\det(U_{4}) = \det\begin{pmatrix} P_{k-1} &\gamma \\ \gamma^{\ast} & N_{\{1, \ldots, k-2, k\}} \end{pmatrix} =
P_{k-1}N_{\{1, \ldots, k-2, k\}} - |\gamma|^{2}.\]

Applying Equation \eqref{detid}, we obtain
\[ N_{k}^{-1}P_{k-2} = P_{k-1}N_{\{1, \ldots, k-2, k\}} - |\gamma|^{2}.\]
The terms $P_{k-2}$ and $P_{k-1}$ are determinants of Gram matrices and hence non-negative.
By the inductive hypothesis, $N_{\{1, \ldots, k-2, k\}} \leq 0$ so the right-hand side is non-positive. The signs of $N_{k}^{-1}$ and
$N_{k}$ agree, and so $N_{k}$ is non-positive and the result is established by induction.

Since the $g_{k,k}$ are real and positive, and the $N_{k}$ are non-positive,
Equation \eqref{recursive} now shows that $P_{k} \leq \prod_{i=1}^{k}g_{i,i}$.
By hypothesis, all entries in $M$ have modulus bounded by $1$ so each term
in the product satisfies $|g_{k,k}| \leq n$ and $\det(G) \leq n^{n}$.
Finally, $|\det(M)| \leq {n}^{n/2}$ and the proof is complete.
\end{proof}

Equality holds in the identity $P_{k} = g_{k,k}P_{k-1} + N_{k}$ if and only if $\det(N_{k})$ vanishes.
Since $N_{k}$ contains a positive definite minor $P_{k-1}$, this occurs if and only if the final column of $N_{k}$
is identically zero. Applying this observation repeatedly, equality in Theorem \ref{Had} holds if and only if all of
the minors $N_{k}$ vanish, which forces $MM^{\ast}$ to be diagonal.

The most substantial part of Hadamard's proof is devoted to establishing that the determinant of a
symmetric positive definite matrix is bounded by the product of its diagonal elements. This more general result
was conjectured by William Thomson (later Lord Kelvin) in 1885. As recounted by Maritz in his masterly
mathematical biography of Thomas Muir \cite{Maritz}, the result was established by Muir shortly afterward.
For reasons never elaborated upon, Muir's publication was delayed until 1901. Even then, the
result is established only for $4 \times 4$ matrices, with the claim that the proof extended easily to
larger dimensions\footnote{Muir returned to this topic in 1910, beginning his paper with the claim that Hadamard's result is
\textit{neither short nor simple, the method being that known as ``mathematical induction''}.
It concludes twelve pages later, having considered the $3 \times 3$ and $4 \times 4$ cases extensively
with a postscript containing yet another proof of the inequality, this time featuring the minors of order
$2$ of a $4\times 4$ matrix, claimed by Muir to be \textit{of the most direct and simple character}
\cite{Muir1910}. Hadamard's proof is preferable, at least to the authors.}. In 1899, Fredholm \cite{Fredholm} also established Thomson's conjecture, but
acknowledged in 1900 that this result was a direct consequence of Theorem \ref{Had}.

In the form of a bound on the determinant of a symmetric positive definite matrix, Hadamard's result gained
importance due to connections to Fredholm's theory of differential equations, with a new proof by Wirtinger
in~1907 \cite{Wirtinger} and a generalisation by Fischer in 1908 \cite{Fischer}. In fact, we shall have use for Fischer's
{i}nequality in later sections of this paper, and so provide a proof modeled closely on the original. Both results in this section are easily established using techniques of positive definite matrices. Bechenbach and Bellman claim that there are \emph{perhaps a hundred proofs} of the Hadamard inequality \cite{BeckBell}; a one-line proof is given on page 505 of Horn and Johnson \cite{HornJohnson}.

\begin{theorem}[Satz III, \cite{Fischer}]\label{Fischer}
Suppose that $G$ is positive definite and symmetric, and that
\[ G =
\begin{pmatrix} A^{\phantom{T}} & B \\ B^{\ast} & D
\end{pmatrix}\!,
 \]
where $A$ and $D$ are square submatrices. Then $\det(G) \leq \det(A)\det(D)$ with equality if and only if $B = \textbf{0}$.
\end{theorem}

\begin{proof}
To fix notation, let $A$ be $k \times k$ and $D$ be $(n-k) \times (n-k)$. We will follow Fischer's proof,
which involves the $k^{\textrm{th}}$ \textit{compound} of $G$. This is the matrix with rows and columns
indexed by the distinct $k$-subsets of $\{1, \ldots, n\}$ with the entry in row $X$ and
column $Y$ the minor $G_{X, Y}$ of $G$
with rows labelled by $X$ and columns labelled by $Y$. We denote the $k^{\textrm{th}}$ compound of a matrix $M$ by $M^{(k)}$. The following results on compounds would have been well known to Fischer's contemporaries (for further discussion, see, for example, Section 0.8 of \cite{HornJohnson}):
\begin{enumerate}[noitemsep,topsep=3pt]
	\item The Sylvester--Franke theorem: $\det(M^{(k)}) = \det(M)^{\binom{n-1}{k-1}}$.
	\item Jacobi's formula for the $k^{\textrm{th}}$ adjugate: $\textrm{Adj}(M^{(k)})_{X, Y}$ is $(-1)^{\sigma(X, Y)} M_{\overline{X}, \overline{Y}}$, where $M_{\overline{X}, \overline{Y}}$ is the complementary minor of $M_{X, Y}$ and $\sigma(X, Y) = \sum_{x \in X} x + \sum_{y \in Y} y$. The $k^{\textrm{th}}$ adjugate satisfies the relation $\textrm{Adj}(M^{(k)})M^{(k)} = \det(M) I_{\binom{n}{k}}$.
	\item The ({g}eneralised) Cauchy--Binet {f}ormula: $(M_{1}M_{2})^{(k)} = M_{1}^{(k)} M_{2}^{(k)}$.
\end{enumerate}

Fischer establishes a Hadamard-type bound for positive definite matrices. In the notation of Theorem \ref{Had},
\begin{equation}\label{HadDeriv} \det(G) \leq P_{n-1} g_{n,n} .\end{equation}
This result is immediate from the proof of Theorem \ref{Had}, which also shows that the bound is
attained precisely when $g_{i,n} = 0$ for $1 \leq i \leq n-1$.

Next, Fischer decomposes the $k^{\textrm{th}}$ compound as
\[ G^{(k)} =
\begin{pmatrix}
F & f \\ f^{\ast} & \det(A)
\end{pmatrix}\!,\]
where $F$ is square of order $\binom{n}{k}-1$, and $f$ is a column vector. We evaluate $\det(F)$ via the method of
Equation~\eqref{SchurComplement}. Set $U = G^{(k)}$ and $U_{1} = F$. By
the Sylvester--Franke theorem, $\det(U) = \det(G)^{\binom{n-1}{k-1}}$. By
Jacobi's formula for the adjugate, $V_{4}$ is proportional to the minor of $G$ complementary to $A$, namely $V_{4} = \det(G)^{-1}\det(D)$.
Hence $\det(F) = \det(G)^{\binom{n-1}{k-1}-1} \det(D)$.

By hypothesis, $G$ is symmetric positive definite, so $G = MM^{\ast}$ for some matrix $M$. By the Cauchy--Binet
formula,
$M^{(k)}(M^{(k)})^{\ast} = G^{(k)}$ and hence $G^{(k)}$ is positive definite. So we may apply
Equation \eqref{HadDeriv} to $G^{(k)}$ to obtain
\[ \det(G^{(k)}) = \det(G)^{\binom{n-1}{k-1}} \leq  \det(A) \det(F) = \det(A)\det(D)\det(G)^{\binom{n-1}{k-1}-1}\]
from which Fischer's {i}nequality follows by cancelling the common factor of $\det(G)^{\binom{n-1}{k-1}-1}$.

If $\det(A) = 0$ then $\det(G)  = 0$ and Fischer's {i}nequality holds trivially, so suppose that $A$ has full rank. The entries of $f$ are minors of $G$ in which the columns of $A$ are held fixed and the rows vary: these are precisely the minors
with rows drawn from $A$ and $B^{\ast}$. For a fixed row $b_{r}$ of $B^{\ast}$ consider the minors consisting of $k-1$ rows of $A$ and $b_{r}$.
All of these minors vanish if and only if $b_{r} = \textbf{0}$. But equality holds in Equation \ref{HadDeriv} precisely when all entries of the vector $f$ are zero; hence $B^{\ast}$ (and $B$) are zero matrices.
\end{proof}

In the original paper, Fischer characterises the cases of equality in Theorem \ref{Fischer} via an argument similar to Hadamard's demonstration that the minors $N_{k}$ are non-positive. We substitute a slightly more direct (if anachronistic) proof. Fischer also provides a direct proof of Equation \eqref{HadDeriv}, so his theorem gives an independent proof of Theorem \ref{Had}. To see this, apply Theorem \ref{Fischer} recursively to the Gram matrix $G = MM^{\ast}$ until $1 \times 1$ blocks on the diagonal are obtained. Then the determinant of $G$ is bounded by the product of its diagonal entries, and the last sentence of the proof of Theorem \ref{Had} completes the proof.

\section{Hadamard matrices and the maximal determinant problem}
\label{SecMaxDet}

Let $G$ be a symmetric positive definite matrix. As we have seen, the key step in Hadamard's proof of Theorem~\ref{Had} is establishing
the bound $\det(G) \leq \prod_{i=1}^{n} g_{i,i}$.
From Hadamard (but more explicitly from Fischer), one sees that that this bound is met with equality precisely  when $G$ is diagonal. When $G =
MM^{\ast}$ is a Gram matrix, we see that the maximal determinant is obtained precisely when the rows of $M$ are orthogonal. Geometrically,
the volume of a parallelopiped with fixed edge lengths is maximised when the edges are orthogonal. This geometric approach was used by Craigen \cite{CraigenICA} to establish Hadamard's inequality directly from Pythagoras.
There is no existence question to consider here: orthogonal matrices are plentiful and rows can be renormalised at will.
As noted already by Sylvester \cite{Sylvester}, the discrete Fourier transform matrices furnish examples which saturate Hadamard's
determinant bound in any dimension over the complex field. In contrast, there is a non-trivial existence theory for matrices saturating Hadamard's determinant bound over $\mathbb{R}$, which we consider in this section.

Suppose now that $M$ is a real-valued $n \times n$ matrix of maximal determinant with entries of norm at most $1$. Since the determinant is a linear function of the matrix entry $M_{i,j}$, without loss of generality, the entries can be chosen from $\{\pm 1\}$. The remainder of this survey is devoted to the following problem, originally suggested as a topic for investigation by Hadamard.

\begin{mdprob}
What is the maximal determinant of an $n \times n$ matrix with entries in $\{\pm 1\}$?
\end{mdprob}

Initial progress on this problem was made by Hadamard, who established the following result.

\begin{proposition}\label{Hadbasics}
Suppose that $H$ is a real matrix saturating the determinant bound. Then:
\begin{enumerate}[noitemsep,topsep=3pt]
	\item All entries of $H$ belong to $\{ \pm 1\}$.
	\item The rows and columns of $H$ are orthogonal.
	\item The order of $H$ is $1$, $2$ or a multiple of $4$.
\end{enumerate}
\end{proposition}
\begin{proof}
The first two claims follow directly from Hadamard's observation that the bound is saturated if and only if $HH^{\top} = nI_{n}$.
For the last claim, observe that the matrices
\[ (1) \quad\textrm{and}\quad \begin{pmatrix} 1 & \phantom{-}1 \\ 1 & -1 \end{pmatrix}\]
saturate the bound in dimensions $1$ and $2$. Suppose that $H$ has dimension $n \geq 3$. Since the magnitude of the determinant
is invariant under permutation and negation of rows and columns, we may assume that the first row of $H$ has all entries positive.
Orthogonality then forces an equal number of positive and negative entries in the second row. Hence $n$ is even.

The proof that $n$ is divisible by $4$ is only slightly more involved. Consider permuting the columns of $H$ so that the first three rows are in the form
\[ \begin{pmatrix} 1_{a} & \phantom{-}1_{b} & \phantom{-}1_{c} & \phantom{-}1_{d} \\ 1_{a} & \phantom{-}1_{b} & -1_{c} & -1_{d} \\ 1_{a} & -1_{b} & \phantom{-}1_{c} &
-1_{d}\end{pmatrix}\!,\]
where $1_{x}$ denotes an all-ones vector of length $x$. Orthogonality of rows forces the equations
\[ a + b - c - d = 0,\quad a - b + c - d = 0,\quad a - b - c + d = 0.\]
These equations are solved precisely when $a = b = c = d$ and hence the dimension is a multiple of $4$.
\end{proof}

Matrices meeting the determinant bound with equality have become known as \textit{Hadamard matrices}. There is a
substantial literature devoted to Hadamard matrices; we refer the reader to three monographs which have appeared in the past 15 years for
further details, \cite{deLauneyFlannery, HoradamHadamard, Seberry17}. Existence of Hadamard matrices is well-studied. The
following omnibus result provides references to some well-known constructions of Hadamard matrices.

\begin{proposition}\label{Hadexist}
Hadamard matrices exist at the following orders.
\begin{enumerate}[noitemsep,topsep=3pt]
	\item $2^{t}$ for $t \geq 0$ {\upshape\cite{Sylvester}}.
	\item $p^{a} + 1$ where $p$ is prime and $p^{a} \equiv 3 \mod 4$ {\upshape\cite{Paley1933}}.
	\item $2(p^{a} + 1)$ where $p$ is prime and $p^{a} \equiv 1 \mod 4$ {\upshape\cite{Paley1933}}.
	\item $p(p+2) + 1$ where $p$ and $p+2$ are twin primes {\upshape\cite{StantonSprott}}.
	\item $4p^{4t}$ where $p$ is prime and $t \geq 1$ {\upshape\cite{XiaMenonDiffSets}}
	\item $4t$ for all values of $t \leq 250$ except for $t \in \{167, 179, 223\}$ {\upshape\cite{Kharaghani428}}.
	\item $n = ab/2$ or $n = abcd/16$ where $a, b, c, d$ are orders of Hadamard matrices {\upshape\cite{CraigenSeberryZhang,SeberryYamadaProducts}}.
	\item There exist constants $\alpha$ and $\beta$ such that, if $t$ is an odd positive integer, then there exists
a Hadamard matrix of order $2^{\lceil \alpha + \beta \log_{2}(t)\rceil }t$; see {\upshape\cite{Craigen, Seberry17}}.
\end{enumerate}
\end{proposition}

As demonstrated, the maximal determinant problem for $n \equiv 0 \mod 4$ is extensive.
Paley conjectured in the 1930s that the bound is attained in every dimension divisible by $4$.
We note that Hadamard matrices have found application in the construction of error-correcting codes,
experimental designs and more recently in the design of quantum algorithms. The reader is referred to the
monographs of Horadam \cite{HoradamHadamard} and Bengtsson and Zyczkowski \cite{KarolBengtsson} for further details.

\subsection{Finite fields, quadratic residues and the Paley construction}\label{SecPaley}

The guiding principle in the assembly of this survey was to produce a self-contained reference on the maximal determinant problem.
Upper bounds are only half of this story. To establish that the bounds are optimal, infinite families of matrices achieving these bounds
are
required. As illustrated in Proposition \ref{Hadexist}, there are many constructions for Hadamard matrices. We shall see in Section
\ref{Sec1mod4}
that there are just two known constructions for infinite families of matrices when $n \equiv 1, 2 \mod 4$ saturating the relevant determinant bounds. All of these constructions rely on properties of
quadratic residues in
finite fields. We will assume the following results about finite fields, proofs of which can be found in a standard
textbook on abstract
algebra, e.g., \cite{IsaacsAlgebra}.

\begin{enumerate}[noitemsep,topsep=3pt]
	\item For each odd prime power $q$ there exists a finite field with $q$ elements, unique up to isomorphism. We denote this field by $\mathbb{F}_{q}$.
	\item The multiplicative group of $\mathbb{F}_{q}$ is cyclic of order $q-1$.
	\item An element $x \in \mathbb{F}_{q}$ is a \textit{quadratic residue} if there exists $y \in \mathbb{F}_{q}$ such that $y^{2} =
x$.
	Otherwise, $x$ is a \textit{quadratic non-residue}. The function $\chi: \mathbb{F}_{q} \rightarrow \mathbb{C}$ given by $\chi(0)  =
0$, $\chi(x) = 1$ if $x$ is a quadratic residue and $\chi(x) = -1$ otherwise is a \textit{multiplicative character} of $\mathbb{F}_{q}$
and $\chi(x) = x^{\frac{q-1}{2}}$. Hence the number of non-zero quadratic residues is $\frac{q-1}{2}$.
		\item It follows that $\chi(-1) = (-1)^{\frac{q-1}{2}}$, so $-1$ is a quadratic residue if and only if $q \equiv 1 \mod 4$.
\end{enumerate}

The matrices constructed in Proposition \ref{paleycore} and their variants are frequently useful in the construction of maximal determinant matrices,
and also occur in multiple other contexts.

\begin{proposition} \label{paleycore}
Suppose that $p$ is an odd prime number and $\chi$
is the quadratic character of $\mathbb{F}_{p}$.
We define $\chi(0) = 0$. Then the \textit{Paley core} matrix
\[ Q = (  \chi(x-y))_{0 \leq x,y \leq p-1} \]
has zeroes on the diagonal and off-diagonal entries in $\{\pm 1\}$.
Further, $Q$ is circulant and satisfies $QQ^{\top} = pI - J$.
\end{proposition}

\begin{proof}
The matrix is circulant since $(x+1) - (y+1) = x-y$. The matrix has zero entries on the diagonal and $\pm 1$ entries off the diagonal
(depending on whether the equation $z^{2} = x-y$ has a solution or not). So it suffices  to compute the inner product of two rows. Since the
number of non-zero quadratic residues equals the number of non-residues, $\sum_{x \in \mathbb{F}_{p}} \chi(x) = 0$.

We compute the inner product of the rows labelled $a$ and $b$. It will be convenient to sum over the non-zero terms in the inner product:
\begin{align*}
\langle r_{a}, r_{b} \rangle & =  \sum_{x\neq a,b} \chi(a-x)\chi(b-x)  \\
& =  \sum_{\substack{y = a-x\\ y\neq a-b, 0}} \chi(y)\chi(b-a+y) \\
& =  \sum_{y\neq a-b, 0} \chi(y)\chi(y)\chi\biggl( \frac{b-a}{y} + 1\biggr) \\
& =  \sum_{y\neq a-b, 0}  \chi\biggl(\frac{b-a}{y} + 1\biggr) \\
&=  - \chi(1).
\end{align*}
In moving from the second line to the third, we used that $\chi$ is multiplicative. In moving from the third line to the fourth, we use
that $\chi(y^{2}) = 1$. In moving form the fourth line to the fifth,
we used that the sum $\sum_{x} \chi(x)$ is
equal to $0$. The terms excluded from
the sum are $\chi(1) + \chi(0)$, but $\chi(0) = 0$, and the result follows.
\end{proof}

The next result is the Paley type I construction of Hadamard matrices. Following well-established conventions, a Hadamard matrix $H$ with is called \emph{skew-symmetric} if $H-I$ is skew-symmetric in the usual sense; $(H-I)^{\top}= -(H-I)$.

\begin{proposition}[Lemma 2, \cite{Paley1933}]
Suppose that $p \equiv 3 \mod 4$ is prime, and let $j_{p}$ denote the column vector of length $p$ of all ones. Then the matrix
\[ M =  \begin{bmatrix} Q+I & -j_{p} \\ j_{p}^{\top} & 1 \end{bmatrix} \]
is a skew-symmetric Hadamard matrix of order $p+1$.
\end{proposition}

\begin{proof}
First observe that $Q^{\top}[x,y] = Q[y,x] = \chi(y-x) = -Q[x,y]$. Hence $Q^{\top} = \chi(-1)Q$.
Since $q \equiv 3 \mod 4$, the matrix $Q$ is skew-symmetric, and
\[ (Q+I) (Q+I)^{\top} = QQ^{\top} + Q + Q^{\top} + I = (p+1)I - J .\]

Since all entries of $M$ are in $\{ \pm 1\}$ it suffices to check that distinct rows of $M$ are orthogonal to verify that $MM^{\top} =
(q+1)I_{q+1}$.
Each non-terminal row contains $1 + \frac{q-1}{2}$ negative entries coming from the last column and the non-residues, and so is orthogonal
to the last row.
The inner product of any two non-terminal rows gains a contribution $+1$ from the last column and a contribution of $-1$ from the
remaining $q$ columns.
\end{proof}

Throughout this survey we describe constructions for primes $p \equiv 3 \mod 4$.
In all cases, the constructions generalise (possibly with minor variations) to all odd prime powers.
Thus the construction of Paley type I matrices is essentially unchanged for prime powers $q \equiv 3 \mod 4$,
though indices are drawn from $\{\mathbb{F}_{q}, +\}$, and the resulting matrix has a block-circulant submatrix, rather than a
circulant submatrix. Then for prime powers $q \equiv 1 \mod 4$, the Paley core is symmetric, and a variant of this construction
gives a Hadamard matrix of order $2q+2$. For analysis of the corresponding matrix of order $p \equiv 1 \mod 4$, see Proposition \ref{CohnProp}.

\section{The Ehlich--Wojtas bound}
\label{SecEW}
We have seen that Hadamard's bound is attained infinitely often, conjecturally in every dimension which is a multiple of $4$.
On the other hand, the proof of Proposition \ref{Hadbasics} shows that in all other dimensions no three $\{\pm 1\}$ vectors
are pairwise orthogonal. In this section, we follow the treatment of Wojtas \cite{Wojtas}
to establish tighter bounds on maximal determinants in
these dimensions. The next lemma will be a key tool in bounding the determinant of a non-diagonal positive definite matrix.

\begin{lemma} \label{lem1}
Let $B$ be the following positive definite symmetric matrix, and assume further that $0 < b \leq | b_{i}|$ for $1 \leq i \leq k{:}$
\[ B = \begin{bmatrix}
m & g_{1,2} & g_{1,3} & \ldots & g_{1,k} & b_{1} \\
g_{2,1} & m & g_{2,3} & \ldots & g_{2,k} & b_{2} \\
\vdots & \vdots & \vdots & \ldots & \vdots & \vdots \\
g_{k,1} & g_{k,2} & g_{k,3} & \ldots & m& b_{k} \\
b_{1}^{\ast} & b_{2}^{\ast} & b_{3}^{\ast} & \ldots & b_{k}^{\ast} & b
\end{bmatrix}\!. \]
Then $\det(B) \leq b(m-b)^{k}$.
\end{lemma}

\begin{proof}
For each $i$ in the interval from $1$ to $k$, subtract $b_{i}/b$ times the last row from the $i^{\textrm{th}}$ row. Similarly, subtract
$b_{i}^{\ast}/b$ times the last column from the $i^{\textrm{th}}$ column. The result is a symmetric matrix $B'$ conjugate to $B$,
which is therefore positive definite:
\begin{equation}\label{Bmat}
 B' = \begin{bmatrix}
\vspace{2pt}
 m- \frac{| b_{1}|^{2}}{b} & g'_{1,2} & g'_{1,3} & \ldots & g'_{1,k} & 0\,\\
g'_{2,1} &m- \frac{| b_{2}|^{2}}{b} & g'_{2,3} & \ldots & g'_{2,k} & 0 \,\\
\vdots & \vdots & \vdots & \ldots & \vdots & \vdots \,\\
g'_{k,1} & g'_{k,2} & g'_{k,3} & \ldots & m- \frac{| b_{k}|^{2}}{b}& 0 \,\\
0 & 0 & 0 & \ldots & 0 & b
\end{bmatrix}\!.
\end{equation}
Clearly, $\det(B') = \det(B) = b \Delta$ where $\Delta$ is the determinant of the $k \times k$ matrix in the upper left of $B'$.
We apply the Hadamard bound (as interpreted for positive definite matrices) and the bound $|b_{i}|b^{-1} \geq 1$ to complete the proof:
\[ \det(B) \leq b \prod_{i=1}^{k} \biggl( m-  \frac{| b_{i}|^{2}}{b}\biggr) \leq b \prod_{i=1}^{k} ( m- | b_{i}|
) \leq b ( m- b )^{k}.\qedhere\]
\end{proof}

The next theorem was established independently by Ehlich \cite{Ehlich} and Wojtas \cite{Wojtas}, via essentially the same argument.
We have followed Wojtas' proof, which is determinant theoretic, in the style of Hadamard.

\begin{theorem} \label{Wojtas}
Let $G$ be an $n\times n$ real positive definite symmetric matrix,
with diagonal entries $m$. Let $b$ be a positive real number such that $b \leq |g_{i,j}|$ for all off-diagonal entries of $G$.
Then
\[ \det(G) \leq (m + nb - b) (m-b)^{n-1} .\]
\end{theorem}

\begin{proof}
Since the determinant is linear in the rows of $G$, we rewrite the determinant as follows:
\begin{equation}\label{SplitA}
\det(G)  = \det \begin{bmatrix} m & g_{1,2} & \ldots & g_{1,n-1} & g_{1,n} \\
g_{2,1} & m & \ldots & g_{2,n-1} & g_{2,n} \\
\vdots & \vdots & \ldots & \vdots & \vdots \\
g_{n-1,1} & g_{n-1,2} & \ldots & m& g_{n-1,n} \\
0 & 0 & \ldots & 0 & m-b
\end{bmatrix} +
\det \begin{bmatrix} m & g_{1,2} & \ldots & g_{1,n-1} & g_{1,n} \\
g_{2,1} & m & \ldots & g_{2,n-1} & g_{2,n} \\
\vdots & \vdots & \ldots & \vdots & \vdots \\
g_{n-1,1} & g_{n-1,2} & \ldots & m& g_{n-1,n} \\
g_{n,1} & g_{n,2} & \ldots & g_{n-1,n} & b
\end{bmatrix}\!.
\end{equation}
Consider the second term on the right-hand side of Equation \eqref{SplitA}: the principal minors of the matrix are positive, so the matrix is positive definite if and only if the determinant is positive. This is Sylvester's characterisation of positive definite matrices (see Theorem 7.2.5, \cite{HornJohnson}), so Lemma \ref{lem1} applies. We obtain the inequality
\begin{equation} \label{recursiveEq}
 \det(G) \leq (m-b) \det(G_{n-1}) + b ( m- b )^{n-1},
\end{equation}
where $G_{n-1}$ is the $(n-1)\times (n-1)$ principal minor of $G_{n}$.
In the case that the second term is non-positive, we obtain
\[ \det(G) \leq (m-b) \det(G_{n-1}) \leq  (m-b) \det(G_{n-1}) + b ( m- b )^{n-1} ,
\]
so this inequality holds in either case.

Finally, we establish the result by induction. Observe that for the case $n = 2$, the result holds:
\[ \det \begin{bmatrix} m & g_{1,2} \\ g_{2,1} & m\end{bmatrix} = m^{2} - | g_{1,2}|^{2} \leq m^{2} - a^{2} = (m+a)(m-a) \]
for any $a \leq |g_{1,2}|$. Now, assume the result holds for $(n-1) \times (n-1)$ matrices, in particular for
the matrix $G_{n-1}$ in Equation \eqref{recursiveEq}. Then
\begin{align*}
\det(G) & \leq   (m-b) ( m + (n-1)b - b)(m-b )^{n-2}  + b ( m- b )^{n-1} \\
& \leq  ( m + nb - b) ( m-b)^{n-1} .
\end{align*}
This completes the proof.
\end{proof}

Later a characterisation of certain matrices meeting the bound of Theorem \ref{Wojtas} will be required.

\begin{corollary} \label{daddydidit}
Let $G$ be an $n \times n$ symmetric positive definite matrix, with diagonal entries $n$ and $|g_{i,j}| \geq b$ for all $i \neq j$.
If $\det(G) = (n + (n-1)b) (n-b)^{n-1}$, then up to permutation and negation of rows and columns,
\[ G = (n-b)I + bJ, \]
where $J$ is the all-ones matrix.
\end{corollary}

\begin{proof}
The bound in Theorem \ref{Wojtas} is attained if and only if the bound
in Lemma \ref{lem1} is attained. This relies on the Hadamard bound, which is
attained only if the displayed matrix $B'$ of Equation \eqref{Bmat} is diagonal.

Suppose there is an off-diagonal entry $g_{i,j}$ of magnitude larger than $|b|$. Without loss of generality, we permute the rows and columns of $G$ so that this entry is in the last column. Negating rows and columns, we may
assume that all entries in the last row and column of $G$ are positive. Then we calculate the determinant in the manner of Equation \eqref{SplitA}. Evaluate the determinant of the rightmost term as in Lemma \ref{lem1}, observing that $|g_{i,j}| > b$ forces a strict inequality. Hence $|g_{i,j}| = b$ for all off-diagonal entries in the matrix.

Tracing the proof of Theorem \ref{Wojtas}
with this matrix, we are led again to Lemma \ref{lem1}, in which the bottom-right entry of $G$ is replaced with $b$.
Subtracting the final row of this matrix from all others results in subtracting~$b$ from all entries in the matrix.
This matrix is diagonal precisely when all off-diagonal entries are equal to $b$, completing the proof.
\end{proof}

\section{The Barba bound and matrices with \texorpdfstring{$n \equiv 1 \mod 4$}{n = 1 mod 4}}\label{Sec1mod4}

The next result was first established by Barba \cite{Barba}, but follows easily from Theorem \ref{Wojtas}. For an overview of the history
of this result, see Neubauer and Radcliffe \cite{NeubauerRadcliffe}.

\begin{corollary}\label{Barba}
Let $M$ be a matrix of odd order with entries in $\{\pm 1\}$.
Then $\det(M) \leq \sqrt{2n-1} (n-1)^{\frac{n-1}{2}}$.
\end{corollary}

\begin{proof}
The diagonal entries in the Gram matrix are $n$ and the minimal magnitude of the off-diagonal entries in the Gram matrix is $1$.
Applying Theorem \ref{Wojtas} with $b = 1$ and $m = n$ we obtain
\[ \det(MM^{\ast}) \leq ( 2n-1 ) ( n-1)^{n-1}.\]
Hence $|\mkern-2mu\det(M)| \leq \sqrt{2n-1} (n-1)^{\frac{n-1}{2}}$.
\end{proof}

We will now work to characterise the Gram matrices which attain the bound of Corollary \ref{Barba}. If $M$ is a $\{ \pm 1\}$ matrix of odd
order, then no two rows of $M$ are orthogonal. It is possible to say a little more.

\begin{proposition} \label{mod4}
Let $M$ be a $\{ \pm 1 \}$ matrix of odd order $n$. There exists a diagonal $\{ \pm 1 \}$ matrix
$D$ such that $N = DM$ satisfies $NN^{\top} \equiv nJ \mod 4$. That is, all inner products
in the normalised matrix $N$ are congruent to $n \mod 4$.
\end{proposition}

\begin{proof}
Define $D$ to be the diagonal $\{\pm 1\}$ matrix which contains a $-1$ in row $i$ if and only if the number of negative entries in row $i$ of $M$ is odd. Then every row of $N = DM$ has an even number of $-1$ entries.

Let $u,v$ be $\{\pm 1 \}$ vectors of length $n$ with $2a$ and $2b$ negative entries respectively. Suppose that the negative entries
coincide at $c$ positions.
Then
\[ \langle u, v\rangle = n - 2 ( 2a - c) - 2( 2b - c) = n - 4(a+b-c) \equiv n \mod 4 .\]
So the proposition holds for the matrix $N$.
\end{proof}

Following Ehlich now, we apply Proposition \ref{mod4} to characterise the $\{\pm 1\}$ matrices  (if any) which meet the bound of Theorem
\ref{Wojtas} with equality. (Wojtas' proof of this result involves a rather lengthy discussion of elementary row operations.)

\begin{theorem} \label{1mod4}
Let $M$ be an $n \times n$ matrix with entries in $\{ \pm 1\}$. If $\det(M)$ meets the bound of Corollary \ref{Barba} with equality then$:$
\begin{enumerate}[itemsep=-2pt]
	\item $2n-1$ is a perfect square and $n \equiv 1 \mod 4$.
	\item Up to permutation and negation of rows and columns, $MM^{\top} = (n-1)I + J$.
\end{enumerate}
\end{theorem}

\begin{proof}
Since $M$ is a matrix with integer entries, $|\mkern-2mu\det(M)| = \sqrt{2n-1} (n-1)^{\frac{n-1}{2}}$ is an integer.
Hence $n$ is odd and $2n-1$ is a perfect square. Thus $2n-1 \equiv 1 \mod 8$, and it follows that $n \equiv 1 \mod 4$.

By Proposition \ref{mod4}, we may assume that all entries in $MM^{\top}$ are congruent to $1 \mod 4$. In particular, the off-diagonal
entries belong to the
set $\{ \ldots, -7, -3, 1, 5, \ldots \}$. Theorem \ref{Wojtas} applies with $b = 1$ if and only if all off-diagonal entries are equal
to $1$. The matrices attaining the bound are characterised in Corollary \ref{daddydidit}.
\end{proof}

Section \ref{RonanMagic} contains an explicit construction for an infinite family of matrices satisfying the conditions of Theorem \ref{1mod4}.
Before describing that construction, we give an easy construction for \textit{near-maximal} determinants (i.e., determinants within a constant
factor of the bound). Define the \textit{excess} of a Hadamard matrix to be the sum of its entries.

\begin{proposition}[\cite{FarmakisKounias,RagPes}] \label{FK}
Let $H$ be a Hadamard matrix of order $n-1$, with excess $e(H)$.
Then
\[ M =  \begin{pmatrix} H & \mathbf{1} \\ -\mathbf{1}^{\top} & 1 \end{pmatrix} \]
satisfies $\det(M) = \det(H) ( 1 + e(H)n^{-1} )$.
\end{proposition}

\begin{proof}
This follows directly from the Schur complement formula (Section 0.8, \cite{HornJohnson}). For any block matrix in which $A$ is invertible,
\[ \begin{pmatrix} I & \mathbf{0} \\ -CA^{-1} & I \end{pmatrix} \begin{pmatrix} A & B \\ C& D \end{pmatrix}
\begin{pmatrix} I & -A^{-1}B \\ \mathbf{0} & I \end{pmatrix} = \begin{pmatrix} A & \mathbf{0} \\ \mathbf{0} & D - CA^{-1}B \end{pmatrix}.\]
Apply this result to $M$, observing that $\mathbf{1}^{\top}H\mathbf{1} = e(H)$.
\end{proof}

It is well known that the maximal excess of a Hadamard matrix of order $n$ is bounded above by $n\sqrt{n}$,
and that equality is achieved if and only if $n = 4t^{2}$ is the square of an even integer, and every row has sum~$2t$~\cite{Best77}. A~Hadamard matrix with constant row sums is called \emph{regular} in the literature.
If there exists such a Hadamard matrix\footnote{A regular Hadamard matrix necessarily has square order,
and is equivalent to the existence of a so-called \textit{Menon--Hadamard} $2$-design. Designs will be discussed further in Section \ref{RonanMagic}.}
then Proposition \ref{FK} gives a matrix of order $4t^{2}+1$ with determinant $(2t+1) (4t^{2})^{2t^{2}}$. This should be compared to the bound of
Corollary \ref{Barba}: upon making the substitution $n = 4t^{2} + 1$ we obtain the bound $\det(M) \leq \sqrt{8t^{2} +1} (4t^{2})^{2t^{2}}$.
Comparing $(2t+1)$ to $\sqrt{8t^{2}+ 1}$ we see that this determinant exceeds $1/\sqrt{2}$ of the Barba bound (and indeed is somewhat better for
small values of $t$). Constructions for infinite families of regular Hadamard matrices are known: there exist regular Hadamard matrices of order $4q^{4}$ for every odd prime power~$q$, and there exists a regular Hadamard matrix of order $16n^{2}$ whenever there exists a Hadamard matrix of order~$4n$ \cite{Meisner,MuzychukXiang}. Orrick and Solomon \cite{OrrickSolomon} have developed a normalisation technique which suggests that Hadamard matrices with large excess are relatively common.

\subsection{Designs and the Brouwer--Whiteman construction}\label{RonanMagic}

In this section, we construct a matrix of order $2p^2 +2p + 1$ satisfying the conditions of Theorem \ref{1mod4}, where
$p \equiv 3 \mod 4$ is prime. This result was obtained independently by Brouwer \cite{Brouwer} and by Whiteman \cite{Whiteman}.
The construction extends readily to all odd prime powers. For the general case, we refer the reader to the work of Neubauer and Radcliffe \cite{NeubauerRadcliffe}. We begin this section by introducing the matrices $I, J$ and $C$ and establishing some of their basic properties.
In Propositions \ref{affine} and \ref{NR2} we combine these ingredients to form large sets of orthogonal vectors in dimensions $p^{2}$ and $p^{2} + 2p$ respectively. Then in Theorem \ref{NRthm}, we add a single row and column to these matrices to yield a maximal determinant matrix in dimension $2p^{2} + 2p + 1$.

Recall that $I$ and $J$ denote the identity and all-ones matrix respectively, where the dimension is clear from context. Let $j_{m}$ denote the row vector of length $m$
with all entries equal to $1$. A useful observation is that for any matrix $M$, the entries of $JM$ are the column sums of $M$ while
the entries of $MJ$ are the row sums of $M$.

Let $Q$ be the $p \times p$ Paley core of Proposition \ref{paleycore}, and let $C = Q-I$. The reader
should verify that $C$ has all entries in $\{\pm 1\}$ and, since $p \equiv 3 \mod 4$, that $Q$ is skew-symmetric, and
\[ CC^{\top} = (Q - I)(Q-I)^{\top} = QQ^{\top} - Q - Q^{\top} + I = (p+1)I - J \,.\]
It follows from Proposition \ref{paleycore} that $JC = CJ = -J$.

Finally, define the tensor product $A \otimes B = [a_{i,j}B]_{i,j}$. Provided the matrices have compatible dimensions,
matrix multiplication distributes over the tensor product: $(A \otimes B) (M \otimes N) = AM \otimes BN$.
We will require some well-known results from the theory of combinatorial designs in this section; for
further information the reader is directed to the monograph of Beth, Jungnickel and Lenz \cite{BJL}.

\begin{definition}
Let  $V$ be a set of size $v$ whose elements are called \textit{points}, and a set $B$ of
\textit{blocks}, each of which is a $k$-subset of $V$.
The pair $(V, B)$ is a $2$-$(v,k,\lambda)$ design if each pair of points is contained in
precisely $\lambda$ blocks. An incidence matrix $M$ of the design $(V,B)$
has rows labelled by points, columns labelled by blocks and $m_{v,b} = 1$ if $v \in b$ and $0$ otherwise.
A matrix $M$ with entries $\{0,1\}$ is the incidence matrix
of a $2$-design if and only if
\[ MM^{\top} = (k-\lambda)I_{v} + \lambda J_{v}.\]
\end{definition}

The affine designs are an important family of $2$-designs obtained from vector spaces over finite fields.

\begin{definition}
Let $U$ be a vector space of dimension $2$ over $\mathbb{F}_{p}$. Let $V$ be the set of vectors of $U$ and
$B$ be the set of $1$-dimensional subspaces and their translates. Since any two vectors determine a unique line,
$(V, B)$ is a $2$-$(q^{2}, q, 1)$ design. The incidence matrix is $q^{2} \times (q^{2} + q)$, and can be partitioned into $q+1$ \textit{parallel classes}:
sets of blocks which partition the point set.
\end{definition}

Let us be a little more explicit in our description of the affine plane: parallel classes consist of pencils of parallel lines in the plane.
One pencil consists of ``vertical'' lines, which are all of the form $\{ (c, x) : x \in \mathbb{F}_{p}\}$ for fixed $c \in \mathbb{F}_{p}$.
The remaining lines consist of point-sets of the form $\{ (x, ax+b) : x \in \mathbb{F}_{p}\}$ for some $a, b \in \mathbb{F}_{p}$.
The parallel classes are obtained by fixing $a$ and varying $b$.

The incidence matrix of the affine plane has $p^{2}$ rows and $p^{2} + p$ columns. We will assume that the columns are grouped into $p+1$ parallel classes. By elementary linear algebra, each $p^{2} \times p$ submatrix contains a unique $1$ in each row, and $p$ non-zero entries in each column. Denote this matrix by $M_{p}$, and observe that $M_{p}M_{p}^{\top} = pI + J$.

\begin{proposition} \label{affine}
Let $M_{p}$ be the incidence matrix of the affine plane of order $p$ and let $C = Q-I$ be the Paley core of order $p$.
Then $M = M_{p} \left( I_{p+1} \otimes C\right)$ is a $p^{2} \times (p^{2}+p)$ matrix with entries in $\{ \pm 1\}$ which satisfies
\[ MM^{\top} = p^{2} I_{p^{2}}.\]
Each row of $C$ occurs $p$ times in each column-block of $M$. Each row of $C$ occurs at least once in each row of $M$.
It will be convenient to write $M$ as a block matrix, which we denote $[ M_{0} \mid M_{1} ]$ where $M_{0}$ consists of a single parallel class.
\end{proposition}

\begin{proof}
Consider the $p^{2} \times p$ submatrix $F$ of $M_{p}$ corresponding to the $i^{\textrm{th}}$ parallel class. The corresponding block of $M$ is just $FC$. Since each row of $F$ contains a single $1$, every row of $FC$ is just a row of $C$. Hence the entries of $M$ all belong to $\pm 1$, and the diagonal entries of $MM^{\top}$ are all $p^{2}$.

By the $2$-design property, any pair of points are contained in a unique block, so the inner product
of two rows in $M_{p}$ is $1$. Hence for any two distinct rows of $M$, there is a unique parallel class in
which they have the same row of $C$. In all other parallel classes they differ. Hence, the inner product
gains a $+p$ term from the parallel class where they agree,
and $p$ terms $-1$ from the parallel classes
in which they disagree, and every pair of rows is orthogonal.
\end{proof}

The next proposition, like the previous one, constructs a large set of orthogonal vectors with rows drawn from $J$ and $C$.

\begin{proposition} \label{NR2}
Let $C$ be the Paley core of order $p$, where $p \equiv 3 \mod 4$. Let $J$ be the all-ones matrix of order~$p$, and let $j_{p}$ be a vector of ones of length $p$.
Then the $(p^{2} + 2p) \times (2p^{2} + 2p)$ matrix
\[
N =
\advance\arraycolsep3pt
\begin{bmatrix}
 -J 					& -C \otimes j_{p} & J & (C+2I) \otimes j_{p} \\
 -j_{p}^{\top} \otimes C 	& -(C+I)\otimes C + I \otimes J & j_{p}^{\top} \otimes C & (C+I)\otimes C + I \otimes J \\
 + J & j_{p} \otimes C 	& -J & j_{p} \otimes C
\end{bmatrix} \]
satisfies $NN^{\top} = (2p^{2} + 2p) I_{p^{2} + 2p}$.
\end{proposition}

\begin{proof}
Essentially, the proof reduces to computing $NN^{\top}$ and carefully evaluating each of the terms.
Let us compute the inner product of the first block of the matrix with itself (equivalently, the inner product of any
two rows from the first block).
First observe that $N$ is a $\{\pm 1\}$ matrix, so the diagonal of $NN^{\top}$ is as claimed. Recall that $CC^{\top} = (p+1)I - J$,
and that since $p \equiv 3 \mod 4$, we have $C + C^{\top} = -2I$.
\begin{align*}
N_{1}N_{1}^{\top} & =
\begin{bmatrix}
-J & -C \otimes j_{p} & J & (C+2I) \otimes j_{p} \end{bmatrix}
\begin{bmatrix} -J & -C \otimes j_{p} & J & (C+2I) \otimes j_{p} \end{bmatrix}^{\top} \\
& =  JJ^{\top} + (C \otimes j_{p} )(C \otimes j_{p})^{\top} + JJ^{\top} + ((C+2I) \otimes j_{p}) ( (C+2I) \otimes j_{p})^{\top}\\
& =  2pJ + CC^{\top} \otimes j_{p}j_{p}^{\top} + (C+2I)(C+2I)^{\top} \otimes  j_{p}j_{p}^{\top}\\
& =  2pJ + p CC^{\top} + p( CC^{\top} + 2C + 2C^{\top} + 4I) \\
& =  2pJ + p( (p+1)I -J ) + p((p+1) I - J)\\
& =  2p(p+1)I.
\end{align*}
In particular, we conclude that two distinct rows from this block are orthogonal. We now verify the orthogonality of rows from two distinct blocks. To perform this computation by hand, it is convenient to simplify each term in the product individually, using that $j_{p} \otimes J = J \otimes j_{p}$, and that $J( j_{p} \otimes C) = j_{p} \otimes JC = - J\otimes j_{p}$:
\def\mk{\mkern-1.7mu}
\begin{align*}
N_{1}N_{2}^{\top} & =
\advance\arraycolsep0.85pt
\begin{bmatrix}
-J & -C \mk\otimes\mk j_{p} & J & (C\mk+\mk2I)\mk \otimes\mk j_{p} \end{bmatrix}\mk
\begin{bmatrix} - j_{p}^{\top}\mk \otimes\mk C & -(C\mk+\mk I)\mk\otimes\mk C\mk +\mk I\mk \otimes\mk J & j_{p}^{\top}\mk \otimes\mk C & (C\mk+\mk I)\mk\otimes\mk C\mk +\mk I\mk \otimes\mk J \end{bmatrix}^{\!\top} \\
& =  - J \otimes j_{p} + ( J - (p+1)I - (p+1) C ) \otimes j_{p} - J \otimes j_{p} + ( J + (p-3)I + (p-1)C - 2C^{\top})\otimes j_{p} \\
& =  -2 J \otimes j_{p} + ( J - (p+1)I - (p+1) C ) \otimes j_{p} + ( J + (p-3) I + (p+1)C - 2C - 2C^{\top} ) \otimes j_{p} \\
& =  (- (p+1)I +(p-3)I -(p+1)C+(p+1)C+4I ) \otimes j_{p}\\
& =  \textbf{0} \otimes j_{p}.
\end{align*}
The remaining verifications are similar and are left for the reader.
\end{proof}

In Propositions \ref{affine} and \ref{NR2}, the assumption that $p \equiv 3 \mod 4$ is necessary.
Using the affine plane, we constructed $p^{2}$ pairwise orthogonal vectors with entries $\{\pm 1\}$ in dimension $p^{2} + p$. For primes $p \equiv 1 \mod 4$ this is impossible, by Proposition \ref{Hadbasics}. Using tensor products, we constructed $p^{2} + 2p$ orthogonal vectors in dimension $2p^{2} + 2p$. To complete our construction of maximal determinant matrices, we assemble $M$ and $N$ into a square matrix of dimension $(p+1)^{2} + p^{2}$.

\begin{theorem}\label{NRthm}
Let $W$ be the following matrix, assembled from the matrices of Propositions \ref{affine} and \ref{NR2} with a single row and column appended:
\begin{equation}\label{NR-mat}
\arraycolsep10pt
\def\arraystretch{1.3}
W =
\begin{bmatrix}
1 			& j_{p} 				& -j_{p^2} 					& j_{p} 		& j_{p^2} \\
j_{p}^{\top}	& -J 					& -C \otimes j_{p} 			& J 			& (C+2I) \otimes j_{p} \\
j_{p^2}^{\top}	& -j_{p}^{\top} \otimes C 	& -(C+I)\otimes C + I \otimes J 	& j_{p}^{\top} \otimes C & (C+I)\otimes C + I \otimes J \\
j_{p}^{\top} 	& J 					& -j_{p} \otimes C 			& J 			& j_{p} \otimes C \\
-j_{p^2}^{\top} 	& -M_{0} 				& -M_{1} 					& -M_{0} 		& -M_{1}
\end{bmatrix}\!.
\end{equation}
Then $WW^{\top} = (2p^{2} + 2p) I + J$, and so $W$ is a maximal determinant matrix. Furthermore, $W$ has constant row sums $2p+1$.
\end{theorem}

\begin{proof}
The displayed rows $2$ to $4$ of $W$ consist of the matrix $N$ of Proposition \ref{NR2} with an initial column of ones added.
The final row of $W$ contains an initial column of ones followed by a submatrix $[-M, -M]$ where $M$ is as in Proposition \ref{affine}.
It follows from these results that all entries of $W$ come from $\{ \pm 1\}$. We must show that the inner product of any two distinct rows is $1$.
Since orthogonality of the rows of $M$ and $N$ has already been established, two tasks remain: to compute the inner product of the intial row with any other row, and to show that the inner product of a row of $[-M, -M]$ with a row of $N$ is equal to $2$.

The inner product of the first row with any other can be computed from the row sums of the component blocks of $W$. The row sums of $C$ are $-1$. Hence each row of $M_{0}$ has sum $-1$ and each row of $M_{1}$ has sum~$-q$. Recall also that the row sum of $u\otimes v$ is the
product of the row sums, and that row sums are linear. For example, the inner product of the first row of $W$ with any row from the third block evaluates as
\[ 1 + (1\cdot (-1)^{2}) + (-1)(0 \cdot (-1) + p) + (1\cdot1(-1)) + (1)(0\cdot(-1) + p) = 1.\]
The remaining verifications are similar, and left to the reader.

In light of the first column, we need to show that the inner product of a row of $[-M,-M]$ with a row of $N$ is~$+2$.
Take for example a row from the first block of $N$. Since the rows of $M_{0}$ all come from $C$, the contributions in the second and fourth displayed columns are $-1$ and $1$ respectively. Since $C$ contains $\frac{p-1}{2}$ entries $+1$ and $\frac{p+1}{2}$ entries $-1$, and the rows of $M_{1}$ are concatenations of rows of $C$, the contribution from the third block is $\frac{p+1}{2} -\frac{p-1}{2}$. The contribution from the final block is also $+1$, and hence the inner product evaluates as
\[
-1 -1 + \Bigl(\frac{p+1}{2} -\frac{p-1}{2}\Bigr) + 1 + \Bigl(\frac{p+1}{2} -\frac{p-1}{2}\Bigr) = 1.
\]
Here, too, we leave the remaining verifications to the reader.
\end{proof}

We note again that this result extends readily to odd prime powers; such a matrix has order $(q+1)^2 + q^2$. There are nine orders $n = 4t+1$ with $n \leq 200$ for which $2n-1$ is a perfect square. Of these, $n = 5, 13, 41$ are sufficiently small that maximal determinant matrices may be found by ad hoc means. Orders $n = 25, 61, 113, 181$ are of the form $q^{2} + (q+1)^{2}$, and so Theorem \ref{NRthm} applies. The remaining two cases are open. For $n = 85$, the Barba bound is $13 \cdot 84^{42}$, while Proposition \ref{FK} produces a matrix with determinant $10 \cdot 84^{42}$. A matrix with a larger determinant, $\frac{501}{49}\cdot 84^{42}$, was constructed by Orrick and Solomon \cite{OrrickWeb}. For $n = 145$, the Barba bound is $17 \cdot 144^{72}$ while Proposition \ref{FK} gives a matrix with determinant $13 \cdot 144^{72}$.

At orders $n \equiv 1 \mod 4$ where the Barba bound cannot be attained, rather less is known. Chadjipantelis, Kounias and Moyssiadis \cite{CKM} gave an analysis of the Gram matrices of maximal determinant matrices at orders~$17$ and $21$, and found explicit matrices of maximal determinant. Their method was extended by Brent, Orrick, Osborn and Zimmerman \cite{BOOZ} to find the Gram matrices of maximal determinant at order $37$. To our knowledge, these are the only cases not covered by Theorem \ref{1mod4} for which the maximal determinant is known. To be entirely explicit: we are not aware of work establishing the maximal determinants at orders $29, 33, 45$ or $49$, and these are the only open cases with $n \equiv 1 \mod 4$ and $n \leq 50$. Computational work by Orrick and Solomon shows that for all orders $n \leq 100$, matrices attaining at least $0.7$ of the Barba bound exist, and can be obtained from Hadamard matrices of large excess using Proposition \ref{FK}.

\section{A refined bound and the case \texorpdfstring{$n \equiv 2 \mod 4$}{n = 2 mod 4}}
\label{Sec2mod4}
The analysis of the case $n \equiv 2 \mod 4$ is a continuation of the techniques developed thus far. The results in this section were obtained by Cohn \cite{Cohn64-1}, Ehlich \cite{Ehlich}, Whiteman \cite{Whiteman} and Wojtas \cite{Wojtas}.

\begin{theorem} \label{2mod4}
Let $M$ be an $n \times n$ matrix with entries $\{ \pm 1\}$ where $n \equiv 2 \mod 4$. Then
\[ \det(M) \leq \left(2n-2\right)\left(n-2\right)^{\frac{n-2}{2}}.\]
If $M$ attains the bound then
\[ MM^{\top} = \begin{pmatrix} (n-2)I + 2J & 0 \\ 0 & (n-2)I + 2J \end{pmatrix}\!,\]
up to permutation and negation of rows of $M$, where all blocks of the Gram matrix are ${n}/{2} \times {n}/{2}$.
\end{theorem}

\begin{proof}
We start with the first statement.
Let $G:=MM^{\top},$ with entries $g_{i,j},$ then $G$ is positive definite and symmetric.
Since $n \equiv 2 \mod 4$ and $M$ has entries in $\{ \pm 1\}$, it follows that
$g_{i,i}=n$ and $g_{i,j}$ is even, for all $1 \leq i,j \leq n$.

If no pair of rows of $M$ are orthogonal then every entry of $G$ has magnitude at least $2$, and Theorem \ref{Wojtas} applies with $b=2$ yielding the required bound
\[ \det(G) \leq \left(3n^{2} - 8n +4\right) (n-2)^{n-2} \leq (2n-2)^{2} (n-2)^{n-2}.\]
Otherwise, $g_{i,j} \equiv 0 \mod 4$ for some $i \neq j$. Up to simultaneous permutation of rows and columns of $G$,
we may assume that $g_{1,j} \equiv 2 \mod 4$ for $1 \leq j \leq k$ and $g_{1,j} \equiv 0 \mod 4$ for $k+1 \leq j \leq n$.
Set
\[ G = \begin{pmatrix} A^{\phantom{T}} & B \\ B^{\top} & D \end{pmatrix}\!, \]
where $A$ is $k \times k$ and $D$ is $(n-k) \times (n-k)$. We claim that all entries of $A$ and $D$ are $2 \mod 4$ and
that all entries of $B$ are $0 \mod 4$. For any $r,s,t$ in the range $1$ to $n$, we have
  \begin{align*}
   g_{r,r}+g_{r,s}+g_{s,t}+g_{t,r}&=\sum_{i}{m_{r,i}^2+m_{r,i}m_{s,i}+m_{s,i}m_{t,i}+m_{t,i}m_{r,i}}\\
  &=\sum_{i}{(m_{r,i}+m_{s,i})(m_{r,i}+m_{t,i})}.
    \end{align*}
Since $m_{i,j} \in \{ \pm 1\}$, each of
the terms $(m_{r,i}+m_{s,i})$ and $(m_{r,i}+m_{t,i})$ is even,
so their product is divisible by~$4$.
Since $g_{r,r} \equiv 2 \mod 4,$ it follows that $g_{r,s}+g_{s,t}+g_{t,r} \equiv 2 \mod 4$.
Setting $t = 1$ and $r, s \leq k$ we see that $g_{s,1} \equiv g_{r,1} \equiv 2 \mod 4$ and hence $g_{r,s} \equiv 2 \mod 4$.
Hence, every entry of $A$ is $2 \mod 4$. Similarly, it can be shown that the entries of $D$ are $2 \mod 4$ and,
exploiting that $G$ is symmetric, that the entries of $B$ are $0 \mod 4$.

Next, we apply Theorem \ref{Fischer} to see that
\[ \det(G) \leq \det(A) \det(D).\]
Since the elements of $A$ and $D$ are all $2 \mod 4$, we can apply the bound of Theorem \ref{Wojtas}
with $m = n$ and $b= 2$:
\begin{align*}
\det(G) & \leq (n+2(n-k)-2)(n-2)^{n-k-1}(n+2k-2)(n-2)^{k-1}\\
 & = ((2n-2)^2-(n-2k)^2)(n-2)^{n-2}.
\end{align*}
This bound is maximised when $n- 2k = 0$, or, equivalently, when $k = n/2$.

The bound is attained when equality holds in both Fischer's inequality, which requires that $B = \textbf{0}$, and in the Ehlich--Wojtas
bound with $b = 2$, characterised by Corollary \ref{daddydidit}.
\end{proof}

A little further work gives a necessary Diophantine condition for the existence of a matrix meeting the bound of Theorem \ref{2mod4}.

\begin{theorem} \label{optimal2mod4}
If $M$ is an $n \times n$ matrix meeting the bound of Theorem \ref{2mod4} with equality, then $2n-2$ is the sum of two squares.
\end{theorem}

\begin{proof}
Suppose that $M$ meets the bound of Theorem \ref{2mod4}. Then there exists a signed permutation matrix 
$P_{1}$ such that $P_{1}MM^{\top}P_{1}^{\top} = G$, where $G$ is the Gram matrix given in the theorem statement. 
By the argument of Theorem \ref{2mod4}, any Gram matrix with determinant equal to $\det(G)$ is similar to $G$ by permutation 
and negation of rows and columns. Because $\det(MM^{\top}) = \det(M^{\top}M)$, there exists another signed permutation matrix $P_{2}$ such that $P_{2}M^{\top}MP_{2}^{\top} = G$. Let $N = P_{1}MP_{2}^{\top}$. Then
\[
NN^{\top} = P_{1}MM^{\top}P_{1}^{\top} = G, \quad N^{\top}N = P_{2}M^{\top}MP_{2}^{\top} = G.
\]
Thus $N$ commutes with $N^{\top}$, and it follows that $N$ commutes with $G$. It will be convenient to write
\[ N = \begin{pmatrix} A & B \\ C & D \end{pmatrix}\!,\]
where all blocks are $n/2 \times n/2$, as established in the proof of Theorem \ref{2mod4}.
We then see that $XJ = JX$ for all $X \in \{ A, B, C, D \}$.
But $XJ$ is constant on rows, while $JX$ is constant on columns. We conclude that $XJ = JX = xJ$,
where all row and column sums of $X$ are equal to $x$. To conclude the proof, consider the matrix product
\[ \begin{pmatrix}
 J & 0 \\ 0 & J
\end{pmatrix}\begin{pmatrix}
 A & B \\ C & D
\end{pmatrix}
\begin{pmatrix}
 A & B \\ C & D
\end{pmatrix}^{\top}\! \begin{pmatrix}
 J & 0 \\ 0 & J
\end{pmatrix}\!.\]
Evaluating the product of the first two and the last two matrices, we obtain
\[ \begin{pmatrix}
 aJ & bJ \\ cJ & dJ
\end{pmatrix}\begin{pmatrix}
 aJ & bJ \\ cJ & dJ
\end{pmatrix}^{\top} =
 \begin{pmatrix}
 (a^{2} + b^{2})J^{2} & (ac +bd)J^{2} \\ (ac+bd)J^{2} & (c^{2}+d^{2})J^{2}
\end{pmatrix}\!.\]
On the other hand, evaluating $NN^{\top}$ first, we obtain
\[  \begin{pmatrix}
 J & 0 \\ 0 & J
\end{pmatrix} \left( \begin{array}{cc} (n-2)I + 2J & 0 \\ 0 & (n-2)I + 2J \end{array}\right) \begin{pmatrix}
 J & 0 \\ 0 & J
\end{pmatrix} = \begin{pmatrix}
 (2n-2)J^{2} & 0 \\ 0 & (2n-2)J^{2}
\end{pmatrix}\!.\]
Equating these expressions, we conclude that $a^{2} + b^{2} = 2n-2$, as required.
\end{proof}

It is possible to continue the argument of Theorem \ref{optimal2mod4} a little further: from $ac = -bd$ and $a^{2} + b^{2} = c^{2} + d^{2}$, it follows that $a = \pm d$ and $b = \mp c$. So matrices attaining the bound of Theorem \ref{2mod4} are intimately related to sums of two squares. The well-known characterisation of Fermat shows that an integer fails to be a sum of two squares if and only if its square-free part is divisible by a prime $p \equiv 3 \mod 4$; see, for example, \cite{IrelandRosen}. From Theorem \ref{NRthm} we obtain matrices meeting the bound of Theorem \ref{2mod4}.

\begin{corollary} \label{cor2mod4}
Let $W$ be a matrix of order $n \equiv 1 \mod 4$ meeting the bound of Theorem \ref{1mod4}. Then
\[ \begin{pmatrix}
 W & \phantom{-}W \\ W & -W
\end{pmatrix} \]
is a matrix of order $2n \equiv 2 \mod 4$ which meets the bound of Theorem \ref{2mod4}.
\end{corollary}

\begin{proof}
Compute the Gram matrix: the diagonal blocks are of the form $2WW^{\top} = (2n-2)I_{n} +2J_{n}$, while the off-diagonal blocks are
$\textbf{0}$.
\end{proof}

Of course, not every maximal determinant matrix arises from Corollary \ref{cor2mod4}.
As observed by Koukouvinos, Kounias and Seberry, a construction of Spence using difference
sets and projective planes yields a second infinite family. Note that $(2q+1)^{2} +1 = 2(2q^{2} + 2q + 1)$.

\begin{theorem}[Theorem 1, \cite{Spence}, Theorem 2, \cite{KKS}]\label{Spence}
For any odd prime power $q$ there
exists a pair of circulant matrices $R$ and $S$ of order $v = q^{2} + q + 1$ with entries $\{ \pm 1\}$ such that
\[ RR^{\top} + SS^{\top} = (2v-2)I_{v} + 2J_{v} .\]
The matrix
\[ \begin{pmatrix}
 R^{\phantom{\top}} & S \\ S^{\top} & - R^{\top}
\end{pmatrix}\]
has maximal determinant. The row-sums of $R$ are all equal to $2q+1$ and the row sums of $S$ are $-1$.
\end{theorem}

For an odd prime power $q$, Corollary \ref{cor2mod4} gives matrices of order $4q^{2} + 4q + 2$
while Theorem \ref{Spence} gives matrices of order $2q^{2} + 2q + 2$. To our knowledge, these are
the only known constructions for infinite families of maximal determinant matrices in dimensions $n \equiv 2 \mod 4$.
The following result, seemingly due to Cohn, provides a denser family of matrices which come within a factor of $2$ of optimality.

\begin{proposition}[Theorem 3, \cite{Cohn64-1}]\label{CohnProp}
Let $q \equiv 1 \mod 4$ be a prime power, and let $Q$ be the matrix obtained from the quadratic residue symbol by $Q_{i,j} =
\left(i-j\right)^{q-1/2}$.
Then the matrix
\[ M = \begin{pmatrix}
 Q+ I & -j_{q} \\ j_{q}^{\top} & 1
\end{pmatrix} \]
has order $n = q+1$ and determinant $n(n-2)^{\frac{n-2}{2}}$.
\end{proposition}

\begin{proof}
Since $q \equiv 1 \mod 4$, we have that $-1$ is a quadratic residue in $\mathbb{F}_{q}$. So $Q$ is symmetric and by Proposition~
\ref{paleycore}, $QQ^{\top} = qI -J$.
In particular, the eigenvalues of $QQ^{\top}$ are $0$ with multiplicity $1$ and $q$ with multiplicity $q-1$. Since $\textrm{Tr}(Q) = 0$,
the eigenvalues of $Q$ are $0$ with
multiplicity $1$, and $\pm \sqrt{q}$ each with multiplicity $\frac{q-1}{2}$.
We compute:
\[ MM^{\top} = \begin{pmatrix}
 (q+1)I_{q} + 2Q & 0 \\ 0 & q+1
\end{pmatrix}\!.\]
So the eigenvalues of $MM^{\top}$ are: $(q+1)$ with multiplicity $2$, and $q + 1 \pm 2\sqrt{q}$ each with multiplicity
$\frac{q-1}{2}$.
Hence
\begin{align*}
 \det(MM^{\top}) & =  (q+1)^{2} ( q +1 + 2\sqrt{q})^{\frac{q-1}{2}}  ( q +1 - 2\sqrt{q})^{\frac{q-1}{2}} \\
 & =  (q+1)^{2} ( 1 + \sqrt{q} )^{q-1} ( 1 - \sqrt{q})^{q-1} \\
 & =  (q+1)^{2} (1 - q )^{q-1} \\
 & =  (q+1)^{2} (q-1)^{q-1}.
\end{align*}
Hence $|\mkern-2mu\det(M) | = (q+1)(q-1)^{\frac{q-1}{2}}$, within a
multiplicative factor of $\frac{q+1}{2q-2} \sim \frac{1}{2}$ of the
bound of Theorem \ref{2mod4}.
\end{proof}

There are several other constructions in the literature for matrices of order $n \equiv 2 \mod 4$ with large determinant. Brent and Osborn \cite{BrentOsbornEJC} consider submatrices of order $n-2$ of a Hadamard matrix of order $n$. Brent, Osborne and Smith \cite{BrentOsbornProb} add two rows and columns to a Hadamard matrix. This work is discussed further in Section \ref{newbounds}.
We conclude this section with an overview of known results for small orders. Computational results by Djokovi\'c and Kotsireas \cite{DK1, DK2} show that a pair of circulant matrices $R, S$ satisfying the identity $RR^{\top} + SS^{\top} \!=\! (2n-2)I + 2J$ exists at all orders $n$ for which $2n-2$ is a sum of two squares up to $n = 198$. As in Proposition \ref{Spence}, such matrices easily yield maximal determinant matrices of order $n$. In contrast to the Diophantine condition for matrices meeting the Barba bound, the condition that $2n-2$ be a sum of two squares is relatively easy to satisfy\footnote{Recall that the only obstruction occurs when the square-free part of $2n-2$ has a prime divisor $p \equiv 3 \mod 4$. For example, for $n = 22$, we find that $2n-2 = 42$ is divisible by $3$.}: the only orders with $n \equiv 2 \mod 4$ with $n \leq 100$ for which $2n-2$ is not a sum of two squares are $n \in \{22, 34, 58, 70, 78, 94\}$.\looseness=1

Recent work of Chasiotis, Kounias and Farmakis \cite{CKF1} addresses the smallest of these cases, $n = 22$. Having identified two matrices with large determinant, they perform an exhaustive search for potential Gram matrices with determinant exceeding those of their examples, finding 25 such matrices. Each of these is excluded from being a Gram matrix, and thus the maximal determinant is established to be $40 \cdot 20^{10}$, with two inequivalent Gram matrices being realisable. This should be compared to the bound $42 \cdot 20^{10}$. To our knowledge, the maximal determinant at any order greater than $22$ satisfying $n \equiv 2 \mod 4$ for which $2n-2$ is \textit{not} a sum of two squares remains open.

\section{Ehlich's analysis of the case \texorpdfstring{$n \equiv 3 \mod 4$}{n = 3 mod 4}}
\label{Sec3mod4}

Ehlich develops a bound for maximal determinants when $n \equiv 3 \mod 4$ through a careful analysis
of the minors of such a matrix. These results were previously translated into English and the analysis sharpened by
Brent, Osborn, Orrick and Zimmerman \cite{BOOZ}, but we include our analysis (which differs slightly from theirs) for the sake of
completeness.

For each integer $1 \leq m \leq n$, define the following set of $m \times m$ matrices:
\[ \mathcal{C}_{m} = \{ M \mid m_{i,i} = n, \; m_{i,j} \equiv 3 \mod 4, \; |m_{i,j} | < n \} .\]
The $m \times m$ minors of an $n \times n$ matrix with entries in $\{\pm 1\}$ all belong to $\mathcal{C}_{m}$,
though the set does not consist exclusively of Gram matrices. We will study the maximal determinant of an element of $\mathcal{C}_{m}$,
via inductive methods of the type that we have seen previously. In contrast to previous proofs, the bounds typically cannot be met
with equality. Denote by $\gamma_{m}$ the maximal determinant of an element of $\mathcal{C}_{m}$.

\begin{proposition}\label{growth}
For all $1 \leq m \leq n-1$, we have $\gamma_{m+1} > (n-3) \gamma_{m}$.
\end{proposition}

\begin{proof}
The proof is by induction. Observe first that
\[ \gamma_{1} = n, \quad \gamma_{2} = \det \begin{pmatrix}
 \phantom{-}n & -1 \\ -1 & \phantom{-}n
\end{pmatrix} = n^{2}-1 > n(n-3) .\]

Suppose that $\gamma_{m} > (n-3)\gamma_{m-1}$, and let $C$ be the following $(m + 1) \times (m+1)$ matrix, chosen such that the
top-left $m \times m$ minor is $\gamma_{m}$, and the last row and column are as displayed:
\[ C = \begin{pmatrix}
 A & a & a \\ a^{\top} & n & 3 \\ a^{\top} & 3 & n
\end{pmatrix}\!.\]
We evaluate the determinant as follows:
\begin{align*}
 \det(C) & =  \det \begin{pmatrix}
 A & a & a \\ a^{\top} & n & 3 \\ 0 & 0 & n-3
\end{pmatrix}
+ \det \begin{pmatrix}
 A & a & a \\ a^{\top} & n & 3 \\ a^{\top} & 3 & 3
\end{pmatrix}\\
& =  (n-3) \gamma_{m} + \det \begin{pmatrix}
 A & a & 0 \\ a^{\top} & n & 3-n \\ 0 & 3-n & n-3
\end{pmatrix}\\
& =  (n-3) \gamma_{m} + ( (n-3) \gamma_{m} - (n-3)^{2} \det(A) ) .
\end{align*}
But $\det(A) \leq \gamma_{m-1}$ by definition, so the second term is (strictly) positive by the induction hypothesis. Hence $\gamma_{m+1} \geq \det(C) \geq (n-3)\gamma_{m}$.
\end{proof}

Next, we show that an element of $\mathcal{C}_{m}$ having maximal determinant has, without loss of generality, all off-diagonal elements from the set $\{ -1, 3\}$.

\begin{proposition}\label{EhlichEntries}
If $\det(C) = \gamma_{m}$ then, without loss of generality, $c_{i,j} \in \{ -1, 3\}$.
\end{proposition}

\begin{proof}
Suppose that $C_{1}$ is a positive definite matrix in $\mathcal{C}_{m}$ with some entry $\alpha \not\in \{-1,3\}$, and that $\det(C_{1}) = \gamma_{m}$. Then up to conjugation by a permutation matrix we may assume that
\[
\def\arraystretch{1.1}
C_{1} = \begin{pmatrix}
 A & a_{1} & a_{2} \\ a_{1}^{\top} & n & \alpha \\ a_{2}^{\top} & \alpha & n
\end{pmatrix}
\]
where $| \alpha | \geq 3$ and we further assume that
\begin{equation}\label{2dets}
\det \begin{pmatrix}
 A & a_{2} \\ a_{2}^{\top} & n
\end{pmatrix} \geq \det \begin{pmatrix}
 A & a_{1} \\ a_{1}^{\top} & n
\end{pmatrix}\!.
\end{equation}
If this does not hold, we may permute the final two rows and columns of $C_{1}$ and replace it with a similar matrix with the required property. By the argument of Proposition \ref{growth}, both matrices of Equation \eqref{2dets} are positive definite. Then let
\[
{\renewcommand\arraystretch{1.1}
C_{2} = \begin{pmatrix}
 A & a_{2} & a_{2} \\ a_{2}^{\top} & n & 3 \\ a_{2}^{\top} & 3 & n
\end{pmatrix}\!.
}
\]
We will show that $\det(C_{2}) \geq \det(C_{1})$, contradicting the assumption that $C_{1}$ has maximal determinant. As before, we use that the determinant is linear in the rows:
{\renewcommand\arraystretch{1.1}
\begin{align*}
 \det (C_{1}) & =  \det \begin{pmatrix}
 A & a_{1} & a_{2} \\ 0& n-3 & 0 \\ a_{2}^{\top} & \alpha & n
\end{pmatrix}
+ \det \begin{pmatrix}
 A & a_{1} & a_{2} \\ a_{1}^{\top} & 3 & \alpha \\ a_{2}^{\top} & \alpha & n
\end{pmatrix} \\
& =  \det \begin{pmatrix}
 A & a_{1} & a_{2} \\ 0& n-3 & 0 \\ a_{2}^{\top} & \alpha & n
\end{pmatrix}
+ \det \begin{pmatrix}
 A & a_{1} & a_{2} \\ a_{1}^{\top} & 3 & \alpha \\ 0  & 0 & n-\alpha^{2}/3
\end{pmatrix} +
 \det \begin{pmatrix}
 A & a_{1} & a_{2} \\ a_{1}^{\top} & 3 & \alpha \\ a_{2}^{\top}  & \alpha & \alpha^{2}/3
\end{pmatrix}\!.
\end{align*}
}

Denote the rightmost term in the expansion above by $R$. We have established that the $(n-1) \times (n-1)$ submatrix at the top-left of $R$ is positive definite. So $R$ is positive definite if and only if its determinant is positive. But the bottom-right $2\times 2$ submatrix of $R$ is degenerate. So by Fischer's inequality, if $R$ were positive definite we would have $\det (R) \leq \det(A) \cdot 0$, which is a contradiction. Thus $\det (R) \leq 0$.

Discarding $\det(R)$ we have an upper bound for $\det(C_{1})$ as follows:
\begin{equation} \label{Cmax}
 \det(C_{1}) \leq (n-3)  \det \begin{pmatrix}
 A & a_{2}\\ a_{2}^{\top} & n
\end{pmatrix} + (n - \alpha^{2}/3) \det\begin{pmatrix}
 A & a_{1} \\ a_{1}^{\top} & 3
\end{pmatrix}\!.
\end{equation}
Compute, in the same fashion, the determinant of $C_{2}$:
{\renewcommand\arraystretch{1.1}
\begin{equation*}
\det(C_{2})  =
 \det  \begin{pmatrix}
 A & a_{2} & a_{2} \\ a_{2}^{\top}& n & 3 \\ 0 & 0& n-3
\end{pmatrix}
+ \det \begin{pmatrix}
 A & a_{2} & a_{2} \\ 0& n - 3 & 0 \\ a_{2}^{\top} & 3 & 3
\end{pmatrix}
+ \det \begin{pmatrix}
 A & a_{2} & a_{2} \\ a_{2}^{\top}& 3 & 3 \\ a_{2}^{\top} & 3 & 3
\end{pmatrix}\!.
\end{equation*}
}
Again the third term vanishes, and the first two may be evaluated as before:
\[ \det(C_{2}) =  (n-3)  \det \begin{pmatrix}
 A & a_{2}\\ a_{2}^{\top} & n
\end{pmatrix} + (n - 3) \det\begin{pmatrix}
 A & a_{2} \\ a_{2}^{\top} & 3
\end{pmatrix}\!.\]
Comparing  this with \eqref{Cmax} and recalling the inequality \eqref{2dets}, we get that $\det(C_{2}) \geq \det(C_{1})$ and this inequality is strict if $|\alpha| > 3$. We conclude that an element of maximal determinant in $\mathcal{C}_{m}$ has entries in the set $\{ -1, 3\}$.
\end{proof}

\begin{definition}
{\upshape
Let $J_{t}$ be the $t \times t$ matrix with all entries equal to $1$.
Define $B_{t} = (n-3) I_{t} + 3J_{t}$ to be an \textit{Ehlich-block} of size $t$.
An \textit{Ehlich-block matrix} is an $n \times n$ matrix with Ehlich-blocks along the diagonal,
and all other entries outside the Ehlich-blocks equal to $-1$. To each Ehlich-block matrix
there is associated a \textit{partition} of $n$, given by the Ehlich-block sizes.
}
\end{definition}

\begin{theorem} \label{blockmatrix}
If $\det(C_{m}) = \gamma_{m}$ then, up to similarity, $C_{m}$ is an Ehlich-block matrix.
\end{theorem}

\begin{proof}
We follow the same proof strategy as in Proposition \ref{EhlichEntries}: we explicitly produce a matrix with a larger determinant from an element of $C_{m}$ which is not an Ehlich-block matrix. Up to simultaneous permutation of rows and columns we may assume that the matrices have the form
{\renewcommand\arraystretch{1.1}
\[ C_{1} =  \begin{pmatrix}
 A & a_{1} & a_{2} & a_{3} \\
a_{1}^{\top} & n & -1  & 3 \\
a_{2}^{\top}& -1 & n & 3 \\
a_{3}^{\top} & 3 & 3 & n
\end{pmatrix}\!, \quad
C_{2} = \begin{pmatrix}
 A & a_{1} & a_{3} & a_{3} \\
a_{1}^{\top} & n & 3  & 3 \\
a_{3}^{\top}& 3 & n & 3 \\
a_{3}^{\top} & 3 & 3 & n
\end{pmatrix}\!.
\]
}
Without loss of generality, we assume that the principal minor obtained from deleting the last row and column of $C_{1}$ is less than or equal to the corresponding principal minor of $C_{2}$. (If not, we relabel the rows of $C_{1}$ and redefine $C_{2}$.) We evaluate the determinant of $C_{1}$ using linearity in the rows:
{\renewcommand\arraystretch{1.1}
\begin{align*} \det C_{1} &=  \det \begin{pmatrix}
 A & a_{1} & a_{2} & a_{3} \\
a_{1}^{\top} & n & -1  & 3 \\
a_{2}^{\top}& -1 & n & 3 \\
0 & 0 & 0 & n-3
\end{pmatrix} +
 \det\begin{pmatrix}
 A & a_{1} & a_{2} & a_{3} \\
a_{1}^{\top} & n & -1  & 3 \\
a_{2}^{\top}& -1 & n & 3 \\
a_{3}^{\top} & 3 & 3 & 3
\end{pmatrix} \\
& =
\det \begin{pmatrix}
 A & a_{1} & a_{2} & a_{3} \\
a_{1}^{\top} & n & -1  & 3 \\
a_{2}^{\top}& -1 & n & 3 \\
0 & 0 & 0 & n-3
\end{pmatrix} +
\det\begin{pmatrix}
 A & a_{1} & a_{2} & a_{3} \\
a_{1}^{\top} & n & -1  & 3 \\
0& 0 & n-3 & 0 \\
a_{3}^{\top} & 3 & 3 & 3
\end{pmatrix} +
 \det\begin{pmatrix}
 A & a_{1} & a_{2} & a_{3} \\
a_{1}^{\top} & n & -1  & 3 \\
a_{2}^{\top}& -1 & 3 & 3 \\
a_{3}^{\top} & 3 & 3 & 3
\end{pmatrix} \\
& =  (n-3)\begin{pmatrix}\!  \det \begin{pmatrix}
 A & a_{1} & a_{2} \\
a_{1}^{\top} & n & -1  \\
a_{2}^{\top}& -1 & n
\end{pmatrix} +
 \det\begin{pmatrix}
 A & a_{1} & a_{3} \\
a_{1}^{\top} & n  & 3 \\
a_{3}^{\top} & 3 & 3
\end{pmatrix}\!\end{pmatrix} + \det\begin{pmatrix}
 A & a_{1} & a_{2} & a_{3} \\
a_{1}^{\top} & n & -1  & 3 \\
a_{2}^{\top}& -1 & 3 & 3 \\
a_{3}^{\top} & 3 & 3 & 3
\end{pmatrix}\!.
\end{align*}
}
As before, the rightmost term in this expression violates Fischer's inequality, but has a positive definite submatrix of order $m-1$,
so has non-positive determinant. Expanding the determinant of $C_{2}$ in the same way gives an expression where each term dominates
the corresponding term of $\det(C_{1})$, completing the proof.
\end{proof}

Having established the maximal determinant of a matrix in the class $\mathcal{C}_{n}$ has the structure of Theorem \ref{blockmatrix},
Ehlich evaluates the determinant in terms of the corresponding partition $n = r_{1} + r_{2} + \ldots + r_{s}$, obtaining
\[ \det(C) = (n-3)^{n-s} \prod_{i=1}^{s} ( n- 3 + 4r_{i} ) \biggl( 1 - \sum_{i=1}^{s} \frac{r_{i}}{n-3+4r_{i}} \biggr).\]
Via a lengthy and intricate analysis, Ehlich obtains the following explicit result.

\begin{theorem}[Satz 3.3, \cite{Ehlich}]\label{EhlichBound}
For $n \equiv 3 \mod 4$, the Ehlich-block matrix of maximal determinant has the following structure:
\begin{enumerate}[itemsep=0.5pt]
	\item The partition of $n$ is into $f(n)$ parts where $f(n) = 5$ for $n = 7, 11$ and $f(n) = 6$ for $11 \leq n \leq 59$ and $f(n) = 7$ for all $n \geq 59$.
	\item Each part has size $\lfloor n/f(n) \rfloor$ or $\lceil n/f(n) \rceil$, and this partition is uniquely determined.
\end{enumerate}

For $n \geq 63$, an explicit upper bound on the maximal determinant of an $n \times n$ matrix $M$ is
\[ \det(MM^{\top}) \leq \frac{4 \cdot 11^{6}}{7^{7}} n(n-1)^{6}(n-3)^{n-7} .\]
\end{theorem}

In fact, no matrices are known which achieve the bound given by Ehlich.
Inspecting the approximations made during the proof, this is perhaps unsurprising:
already in the $n = 2$ case of Proposition \ref{growth}, the approximations are not sharp.
Detailed but elementary analysis of the proof of Theorem \ref{EhlichBound} shows that equality in the bound could be achieved if
and only if $n = 7m$. Cohn \cite{Cohn2000} has shown using number theoretic techniques that the Ehlich bound is integral only when $n$ is of
the form $112t^{2} \pm 28t + 7$, while Tamura \cite{Tamura} has applied the Hasse--Minkowski criteria for equivalence of quadratic forms to show
that the smallest order at which the Ehlich bound could be achieved is at least $511$.
On the other hand, Ehlich's bound is asymptotically optimal up to some constant factor.

Orrick \cite{Orrick15} attributes the solution of the maximal determinant problem at orders $n = 3, 7$ to Williamson and $n = 11$ to Ehlich.
In the same paper, Orrick determines the maximal determinant of order $15$. The corresponding Gram matrix has three Ehlich-blocks of size
$4$ and one of size $3$. Later work of Brent, Osborn, Orrick and Zimmermann \cite{BOOZ} computed
the maximal determinant at order $19$. At both orders, the technique used is a careful refinement
of the method of Chadjipantelis, Kounias and Moyssiadis \cite{CKM}: a candidate matrix with large determinant is identified, its Gram matrix is computed, and all symmetric positive definite matrices with larger determinant are ruled out as Gram matrices. Interestingly,
at order $19$, the matrices with largest determinant are \emph{not} Ehlich-block matrices though they contain $18\times 18$ submatrices which are in Ehlich-block form. Bounds on the maximal determinant for $n \equiv 3 \mod 4$ are described in Table \ref{3mod4tab} at the end of the paper.

\subsection{Improved lower bounds for \texorpdfstring{$n \equiv 3 \mod 4$}{n = 3 mod 4}\label{newbounds}}

We conclude with an investigation of direct constructions for $\{\pm 1\}$ matrices with $n \equiv 3 \mod 4$
having large determinant. First we describe results of Brent, Osborn and Smith using the probabilistic method.
Recall that in Proposition \ref{FK}, a Hadamard matrix was augmented by a row and
column of 1’s to obtain a matrix with $n \equiv 1 \mod 4$ and large determinant.
Even when using the optimal Hadamard matrices for this method (those with maximal excess),
the ratio of the determinant obtained to the bound of Corollary \ref{Barba} tends to zero as $n$ tends to infinity.
A remarkable generalisation of this result was obtained by Brent, Osborn and Smith \cite{BrentOsbornProb}, in which
multiple rows and columns are added to a Hadamard matrix. Columns are chosen uniformly at
random, while the rows added are chosen deterministically. Via careful analysis, the authors
show that the ratio of the determinant to the Hadamard bound does \textbf{not} tend to $0$
as $n$ tends to infinity. The reader is referred to the original paper for the proof of the following result.

\begin{theorem}[Theorem 3.6, \cite{BrentOsbornProb}]\label{BrentOsbornThm}
If $0 \leq d \leq 3$, and $h$ is the order of a Hadamard matrix then there exists a matrix $M$ of order $n = h+d$ such that
\[
\left( \frac{2}{e\pi}\right)^{d/2} n^{n/2} \leq \det(M) \leq n^{n/2}\,.
\]
\end{theorem}

A more general result is possible in which the parameter $d$ is not bounded, but all results obtained by these methods contain a factor $(2/e\pi)^{d/2}$. Thus results obtained by this method decay exponentially in the distance to the nearest Hadamard matrix, but are independent of the order of the matrix. In the case $n \equiv 3 \mod 4$, we set $d = 3$ in Theorem \ref{BrentOsbornThm} to obtain a constant $0.1133$. But this comparison is to the Hadamard bound: as $n \rightarrow \infty$ the ratio of the Ehlich and Hadamard bounds tends to $0.4284$, so that for sufficiently large $n$, Theorem~\ref{BrentOsbornThm} shows that whenever there exists a Hadamard matrix of order $n$ there exists a matrix of order $n+3$ achieving at least $0.264$ of the Ehlich bound. As a special case of this result, we highlight the following.

\begin{corollary} \label{BrentCor}
If $p \equiv 3 \mod 4$ is a prime, then there exists a $\{\pm 1\}$ matrix of order $p+4$ which
achieves $0.264$ of the Ehlich bound.
\end{corollary}

Now we analyse two constructions which have appeared in the literature: a construction of
Orrick, Solomon, Dowdeswell and Smith \cite{OSDS} using skew-Hadamard matrices (though we state the result only for Paley cores);
and a generalisation, inspired by Proposition \ref{FK}, of a construction of Neubauer and Radcliffe \cite{NeubauerRadcliffe}.

\begin{proposition} \label{OSDS}
Let $Q$ be the Paley core matrix of order $q$, let
\[ R = j_{q} \otimes
\begin{pmatrix} 1 & -1 & \phantom{-}1 & -1 \\ 1 & \phantom{-}1 & -1 & -1 \\ 1 & -1 & -1 & \phantom{-}1 \end{pmatrix}\!,
\quad
H_{4} =
\begin{pmatrix}
\phantom{-}1 & \phantom{-}1 & \phantom{-}1 & -1 \\
\phantom{-}1 & \phantom{-}1 & -1 & \phantom{-}1 \\
\phantom{-}1 & -1 & \phantom{-}1 & \phantom{-}1 \\
-1 & \phantom{-}1 & \phantom{-}1 & \phantom{-}1
\end{pmatrix}
\]
and let $P = Q \otimes H_{4} - I_{q} \otimes J_{4}$. Then
\vspace{-2pt}
\[ M = \begin{pmatrix} P & R^{\top} \\ R & J_{3} \end{pmatrix} \]
is a matrix of order $4q+3$ with $\det(MM^{\top}) = 16(4q)^{3q+3}(4q+16)^{q-1}$.
\end{proposition}

\begin{proof}
Let $T = 4qI_{4q} + 4I_{q} \otimes J_{4} - J_{4q}$ and $v = j_{q} \otimes (3, -1, -1,-1)$.
The Gram matrix of $M$ has the form
\[ MM^{\top} =
\begin{pmatrix}
T & v^{\top} & v^{\top} & v^{\top} \\
v & 4q+3 & 3 & 3  \\
v & 3 & 4q+3 & 3  \\
v & 3 & 3 & 4q+3
\end{pmatrix}\!.\]
Subtract row $4q+2$ from row $4q+3$, then subtract row $4q+1$ from row $4q+2$, and similarly for columns. Then use linearity of the determinant in row $4q+1$:
\[\det(MM^{\top}) = \det
\begin{pmatrix}
T & v^{\top} & \textbf{0} & \textbf{0} \\
\textbf{0} & 4q+3 & -4q & 0  \\
\textbf{0} & -4q & 8q & -4q  \\
\textbf{0} & 0 & -4q & 8q
\end{pmatrix} +
\det \begin{pmatrix}
T & v^{\top} & \textbf{0} & \textbf{0} \\
v & 0 & 0 & 0 \\
\textbf{0} & -4q & 8q & -4q \\
\textbf{0} & 0 & -4q & 8q
\end{pmatrix}\!.\]
Since these matrices are respectively block-upper triangular and block-lower triangular, the determinant may be evaluated as follows:
\vspace{-5pt}
\[ \det(MM^{\top}) = (64q^{3} + 144q^{2}) \det(T) + 48q^{2} \det \begin{pmatrix} T & v^{\top} \\ v & 0 \end{pmatrix}\!.\]

Standard techniques suffice to evaluate the determinant of $T$, which is $\det(T) = 16(4q)^{3q}(4q+16)^{q-1}$. The determinant of the bordered matrix may be computed via the Schur complement method\footnote{It would be remiss of the authors to finish this survey without commenting on the practical evaluation of determinants. Computing the rank of $T - 4qI$ easily gives the multiplicity of $4q$ as an eigenvalue, for example, and the remaining factors of the determinant are only slightly more difficult to guess and verify. Via Cauchy interlacing, one sees that $(4q)^{3q-2}(4q+16)^{q-2}$ divides the determinant of the bordered matrix. The quotient is a polynomial function of degree at most $5$. Evaluating the determinants of a few small matrices computationally and solving a polynomial interpolation problem, the result follows. For much more on the evaluation of determinants see the work of Krattenthaler \cite{Krattenthaler}.}, evaluating to~$(-3)\det(T)$.
The result follows:
\vspace{-3pt}
\[ \det(MM^{\top}) = (4q)^{3}\det(T) = 16(4q)^{3q+3}(4q+16)^{q-1}\,.\qedhere\]
\end{proof}

In line with previous theorems, we state Proposition \ref{OSDS} for Paley cores, but the result holds more generally whenever there exists a skew-Hadamard matrix of order $q+1$. In particular this holds for $q=15$.

In dimension $4q+3$, the Ehlich bound takes the form $4 \cdot 11^{6}\cdot 7^{-7} (4q+3)(4q+2)^{6}(4q)^{4q-4}$.
Cancelling common factors, the ratio of the determinant of the Gram matrix of Orrick, Solomon, Dowdeswell and Smith to the Ehlich bound is
\[
\frac{28^{7}q^{7}}{11^{6}(4q+4)(4q+2)^{6}(q+4)} \frac{(q+4)^{q}}{q^q}\,.
\]
Taking the limit as $q \rightarrow \infty$, the second fraction tends to $e^4$, while the first tends to $0$. For small prime powers the construction yields matrices remarkably close to the Ehlich bound. Some explicit computations are given in Table \ref{3mod4tab}.

We now begin the analysis of the second construction, described in Proposition \ref{border3mod4}.

\begin{lemma} \label{lem2}
For real numbers $a, b, c, d$, the eigenvalues of the $2k \times 2k$ matrix
	\[ M =  \begin{pmatrix}
 aI + b J & cJ \\ cJ & aI + dJ
\end{pmatrix} \]
are $kr_{1} + a$ and $kr_{2} + a$ with multiplicity $1$ where the $r_{i}$ are the roots of the equation $\lambda^{2} - (b+d)\lambda + (bd-c^{2})$, and the eigenvalue $a$ with multiplicity $2k-2$. This implies $\det(M) = ( a^{2} + ak(b+d) + k^{2}(bd-c^{2}) ) a^{2k-2}$.
\end{lemma}

\begin{proof}
 If
\[  \begin{pmatrix}
 b  & c \\ c & d
\end{pmatrix}  \begin{pmatrix}
 x_{1} \\ x_{2}
\end{pmatrix} = \lambda  \begin{pmatrix}
 x_{1} \\ x_{2}
\end{pmatrix}, \quad {\textrm{then } } \begin{pmatrix}
 b J & cJ \\ cJ & dJ
\end{pmatrix} \begin{pmatrix}
 x_{1} j_{k}\\ x_{2}j_{k}
\end{pmatrix} = k\lambda \begin{pmatrix}
 x_{1} j_{k}\\ x_{2}j_{k}
\end{pmatrix}\!.\]
Since $M - aI_{2k}$ clearly has rank $2$, all other eigenvalues are zero. The eigenvalues of $M$
are of the form $a + \lambda$ where $\lambda$ is an eigenvalue of $M - aI$, so the result follows.
The determinant evaluation follows by identifying the sum and product of the eigenvalues with the
trace and determinant of the $2 \times 2$ matrix, respectively.
\end{proof}

The proof of the next result is identical for the displayed matrices. The matrices of Corollary \ref{cor2mod4}
are in the form of matrix $M_{1}$ while those of Theorem \ref{Spence}, and those constructed by Djokovi\'c and Kotsireas
are in the form of matrix $M_{2}$.

\begin{proposition}\label{border3mod4}
Suppose that $R$ and $S$ are $k \times k$ matrices satisfying the identities
\[ RJ = JR = rJ, \quad SJ = JS = sJ, \quad RR^{\top} + SS^{\top} = (2k-2)I + 2J .\]
Let
\[
{\renewcommand\arraystretch{1.1}
M_{1} = \begin{pmatrix}
R & S & j_{k}^{\top} \\
S & -R & -j_{k}^{\top} \\
j_{k} & j_{k} & 1
\end{pmatrix}\!, \quad
M_{2} = \begin{pmatrix}
R & S & j_{k}^{\top} \\
S^{\top} & -R^{\top} & -j_{k}^{\top} \\
j_{k} & j_{k} & 1
\end{pmatrix}\!.
}
\]
Then
\[ \det(M_{i}M_{i}^{\top}) = ( 4k^{2}r^{2}- 16 k^2 r + 16k^{2} - 16k + 8kr + 4 ) (2k-2)^{2k-2}
\]
with the condition that $RS^{\top} = SR^{\top}$ for $M_{1}$ and no additional condition for $M_{2}$.
\end{proposition}

\begin{proof}
Given the hypotheses, it may be computed directly that
\[ M_{i}M_{i}^{\top} =
\begin{pmatrix}
(2k-2)I + 3J & -J & (1+r+s)j_{k}^{\top}\\
-J& (2k-2)I + 3J & (-1-r+s)j_{k}^{\top}\\
(1+r+s)j_{k} & (-1-r+s)j_{k} & 2k+1
\end{pmatrix}\!.\]
Subtracting multiples of the last row, we clear the last column:
\[
\begin{pmatrix}
(2k-2)I + bJ & cJ & \textbf{0}\\
cJ& (2k-2)I + dJ & \textbf{0}\\
(1+r+s)j_{k} & (-1-r+s)j_{k} & 2k+1
\end{pmatrix}\!,
\]
where $b = 3 - \frac{(1+r+s)^{2}}{2k+1}$,\; $c  = -1- \frac{(1+r-s)(1+r+s)}{2k+1}$ and $d = 3 - \frac{(-1-r+s)^{2}}{2k+1}$.
Applying Lemma \ref{lem2} to the sub-matrix complementary to the last row and column with $a = 2k-2$ and $b, c,d $ as given,
with simplification performed in MAGMA \cite{MAGMA}, we obtain the following factorisation of the determinant:
\[ ( 48k^3 - 8k^2 r^2 - 16 k^2 r - 12k^2 s^2 - 24 k^2 + 4k r^2 + 8kr + 4ks^2 - 8k + 4 ) (2k-2)^{2k-2}.\]
Recall that $r^{2} + s^{2} = 4k-2$, and eliminate the $s^{2}$ terms:
\begin{align*} \det(M_{i}M_{i}^{\top}) &=
 ( 48k^3 - 12k^2 (r^2 +s^{2}) + 4k^{2}r^{2}- 16 k^2 r - 24 k^2 + 4k( r^2 + s^{2}) + 8kr - 8k + 4 ) (2k-2)^{2k-2}  \\
& =  ( 4k^{2}r^{2}- 16 k^2 r + 16k^{2} - 16k + 8kr + 4 ) (2k-2)^{2k-2}. \qedhere
\end{align*}
\end{proof}

In Proposition \ref{border3mod4} the result appears asymmetric in $r$ and $s$.
In fact, from a pair of matrices $R, S$ satisfying $RR^{\top} + SS^{\top} = (2k-2)I + 2J$,
four different determinants are obtained, depending on the row-sum of the matrix on the principal diagonal,
which is drawn from $\{ \pm r, \pm s\}$. For sufficiently large values of
$k$, the terms $4k^{2}r^{2} - 16k^{2}r$ dominate and the determinant is
maximised when $r$ is large and
\vadjust{\goodbreak}
negative.

\begin{theorem} \label{improvedNR}
Let $M$ be a matrix of order $n=2k+1$ as in Proposition \ref{border3mod4}, with $r^{2} + s^{2} = 4k-2$. Then $\det (M)$ achieves a fraction at least ${r^{2}}/{3n}$ of the Ehlich bound.
\begin{itemize}
	\item A matrix exceeding $0.34$ of the Ehlich bound exists of order $n = 4q^{2} + 4q + 3$ for each prime power $q \geq 379$. A matrix exceeding $\frac{1}{3}$ of the bound exists for each $q \geq 47$.
	\item A matrix exceeding $0.48$ of the Ehlich bound exists of order $n = 2q^{2} + 2q + 3$ for each $q \geq 233$. A matrix exceeding $0.47$ of the bound exists for each $q \geq 43$.
\end{itemize}
\end{theorem}

\begin{proof}
In terms of $k$, the Ehlich bound is $4 \cdot 11^{6} \cdot 7^{-7} (2k+1)(2k)^{6}(2k-2)^{2k-6}$.
Let $r$ and $s$ be the constant row sums of a matrix achieving the Ehlich--Wojtas bound. Without loss of generality, we may
assume that $|r| \geq |s|$ and $r < 0$. Since $r^{2} + s^{2} = 4k-2$, it follows that $-\sqrt{4k-2} \leq r \leq -\sqrt{2k-1}$. It will be convenient to write $r = -\alpha \sqrt{2k-1}$ where $1 \leq \alpha \leq \sqrt{2}$. So Proposition~ \ref{border3mod4}
gives a matrix with determinant bounded below by
\begin{align*}
\det(M) & =  ( 4k^{2}r^{2}- 16 k^2 r + 16k^{2} - 16k + 8kr + 4 ) (2k-2)^{2k-2} \\
& \geq  (4\alpha^{2}k^{2}(2k-1) + 16\alpha k^{2}\sqrt{2k-1} + 16k^{2} - 16 k - 8\alpha k\sqrt{2k-1} +4 ) (2k-2)^{2k-2} \\
& \geq  (8\alpha^{2}k^{3} + 16\alpha \sqrt{2}k^{5/2} )(2k-2)^{2k-2},
\end{align*}
where moving from the first line to the second we use that $4k^{2}r^{2}- 16 k^2 r + 8kr$ grows as $r \leq -\sqrt{2k-1}$ tends towards $-\sqrt{4k-2}$. Moving from the second line to the third we observe that the sum of the discarded terms is positive and increasing for all $k \geq 4$.
Cancelling common factors, the ratio to the Ehlich bound is at least
\[
\advance\thinmuskip-1mu\advance\thickmuskip-1mu\medmuskip2mu
\frac{  (8\alpha k^{3} + 16\sqrt{2}\alpha k^{5/2} )(2k-2)^{4} }{ \frac{4 \cdot 11^{6}}{7^{7}} (2k+1)(2k)^{6}} =
\frac{7^{7}}{4 \cdot 11^{6}}\frac{  \left(\alpha^{2}k^{3} + 2\sqrt{2}\alpha k^{5/2} \right)(k-1)^{4} }{  k^{7}+(1/2)k^{6} } =
\frac{7^{7}}{4 \cdot 11^{6}}\frac{  \alpha^{2} k^{7} + 2\sqrt{2}\alpha k^{13/2} - 4k^{6} + O(k^{11/2}) }{  k^{7}+(1/2)k^{6} }.\]
Setting $\alpha = 1$ corresponds to row sums $r = s$ in the maximal determinant matrix of order $n \equiv 2 \mod 4$. Taking the limit as $k \rightarrow \infty$ gives $\frac{7^{7}}{4\cdot 11^{6}} \sim 0.1162$ which is a ratio of the determinants of Gram matrices. Taking a square root gives the claimed lower bound $\sqrt{7^{7} \cdot 2^{-2} \cdot 11^{-6}} \sim 0.34$. A computation shows that the ratio exceeds $\frac{1}{3}$ for $n \geq 8563$ and exceeds $0.34$ for $n \geq 569659$. The bound $4q^{2} + 4q + 3 \geq n$ holds for prime powers $q \geq 47$ and $q \geq 379$ respectively.

Setting $\alpha = \sqrt{2}$ corresponds to setting $r \sim \sqrt{4k-2}$ while $s$ is bounded. Evaluating the displayed equation yields a determinant achieving $\sqrt{7^{7}\cdot2^{-1} \cdot11^{-6}} \sim 0.48$ of the Ehlich bound. The row sums of Corollary \ref{cor2mod4} satisfy $r = s$, while those of Theorem \ref{Spence} satisfy $r^{2} = 4k-3$ and $s^{2} = 1$. A computation shows that the ratio exceeds $0.47$ for $n \geq 3571$ and exceeds $0.48$ for $n \geq 106357$. The bound $2q^{2} + 2q + 3 \geq n$ holds for prime powers $q \geq 43$ and $q \geq 233$ respectively.
\end{proof}

Note that while the constant of Corollary \ref{BrentCor} is smaller than that obtained in Theorem \ref{improvedNR}, the set of orders at which these matrices exist is much denser.

\begin{table}
\begin{center}
\begin{tabular}{cccccc}
$n$ & Upper Bound 							& KMS \cite{KMS}	 	& Prop \ref{OSDS} 	& Prop \ref{border3mod4} 	& Computation \\
\hline
 23 & $\sqrt{45} \cdot 22^{11}$					& $0.3882$	& - 			& - 			& $0.7091$	\\
 27 & $\sqrt{53} \cdot 26^{13}$  				& $0.3600$	& -			& $0.3639$	& $0.7359$	\\
 31 & $\sqrt{61} \cdot 30^{15}$					& $0.3371$	& $0.7060$	& $0.4354$	& $0.7278$	\\
 35 & $\sqrt{69} \cdot 34^{17}$ 					& $0.3181$	& -			& -			& $0.7141$	\\
 39 & $\sqrt{77} \cdot 38^{19}$ 					& $0.3020$	& - 			& $0.3853$	& $0.7253$\\
 43 & $\sqrt{85} \cdot 42^{21}$ 					& $0.2881$	& -			& $0.4477$	& $0.7358$\\
 47 & $\sqrt{93} \cdot 46^{23}$ 					& $0.2760$	& $0.7035$	& $0.4273$	& $0.7035$ \\
 51 & $\sqrt{101} \cdot 50^{25}$ 				& $0.2653$	& -			& $0.3347$	& $0.6481$ \\
 55 & $\sqrt{109} \cdot 54^{27}$ 				& $0.2557$	& -			& $0.3936$	& $0.6544$ \\
 59 & $\sqrt{117} \cdot 58^{29}$ 				& $0.2471$	& -			& -			& $0.7351$ \\
 63 &$\mu \cdot 63^{1/2}\cdot 62^3 \cdot 60^{28}$ 	& $0.2878$	& $0.8146$	& $0.5216$	& $0.9662$ \\
 67 &$\mu \cdot 67^{1/2}\cdot 66^3 \cdot 64^{30}$	& $0.2808$	& -			& $0.4296$	& $0.8635$ \\
 71 &$\mu \cdot 71^{1/2}\cdot 70^3 \cdot 68^{32}$	& $0.2742$	& - 			& - 			& $0.8804$ \\
 75 &$\mu \cdot 75^{1/2}\cdot 74^3 \cdot 72^{34}$	& $0.2608$	& - 			& $0.4834$	& $0.8613$ \\
 79 &$\mu \cdot 79^{1/2}\cdot 78^3 \cdot 76^{36}$	& $0.2623$	& $0.7921$	& - 			& $0.8591$ \\
 83 &$\mu \cdot 83^{1/2}\cdot 82^3 \cdot 80^{38}$	& $0.2569$	& -			& $0.3909$	& $0.8561$ \\
 87 &$\mu \cdot 87^{1/2}\cdot 86^3 \cdot 84^{40}$	& $0.2517$ 	& -			& $0.5222$	& $0.8527$ \\
 91 &$\mu \cdot 91^{1/2}\cdot 90^3 \cdot 88^{42}$	& $0.2469$	& -			& $0.5117$	& $0.8501$ \\
 95 &$\mu \cdot 95^{1/2}\cdot 94^3 \cdot 92^{44}$	& $0.2424$	& $0.7653$	& -			& $0.8447$ \\
 99 &$\mu \cdot 99^{1/2}\cdot 98^3 \cdot 96^{46}$	& $0.2380$	& - 			& $0.4925$	& $0.8496$
 \end{tabular}
\caption{Large determinants with $n \equiv 3 \mod 4$, where $\mu = \sqrt{4\cdot 11^6 \cdot 7^{-7}}$.}
\label{3mod4tab}
\end{center}
\end{table}

We conclude with a table of large determinants for $n \equiv 3 \mod 4$, with $23 \leq n \leq 99$. Following Brent and Yedidia \cite{BrentYedidia}, we display the Barba bound for $n \leq 59$ and the Ehlich bound for $n \geq 63$. In all cases we report the ratio of the determinant of the constructed matrix with the bound given in the second column.

The construction of Koukouvinos, Mitrouli and Seberry \cite{KMS} uses minors of Hadamard matrices and symmetric designs\footnote{Koukouvinos, Mitrouli and Seberry provide two bounds in orders of the form $n = 4t^{2}$. From a maximal minor of a
normalised Hadamard matrix, one obtains a matrix with determinant $(4t^{2})^{2t^{2} - 1}$. From the $\{\pm 1\}$-incidence matrix of a design with parameters $(4t^{2}, 2t^{2}+t, t^{2} + t)$ one obtains a matrix with determinant $2t \cdot (4t^{2})^{2t^{2} - 2}$. The first bound always exceeds the second, but a mis-transcription of the second bound appears in Table 2 of \cite{KMS}.}. We use Proposition \ref{OSDS} to obtain the entries of the fourth column and the results of Djokovi\'c and Kotsireas \cite{DK2}, together with Proposition \ref{border3mod4}, to obtain the entries of the fifth column.

For many years, Orrick maintained a webpage listing the largest known determinant at orders up to $120$. While this page is no longer available it can be accessed via the WayBack Machine \cite{OrrickWeb}. The entries in the last column of Table \ref{3mod4tab} are drawn from this source. For $n \leq 59$, the bound given in Table \ref{3mod4tab} is not the best known, so the ratios computed in the final column differ from the values computed by Orrick.

\section*{Acknowledgements}

The authors would like to thank Michael Tuite and James Ward for helpful discussions on determinant theory,
and to thank Rob Craigen and Guillermo Nu\~{n}ez Ponasso for comments on a draft of this paper. The authors acknowledge Will Orrick, who researched the history of the maximal determinant problem extensively, and made his findings available on his webpage. The authors thank the anonymous referees for their comments which helped us to improve the exposition of this paper.

\bigskip
\noindent
\textsc{Patrick Browne}\\
Technological University of the Shannon: Midlands Midwest, Ireland\\
\href{patrick.browne@tus.ie}{patrick.browne@tus.ie}
\medskip

\noindent
\textsc{Ronan Egan}\\
Dublin City University\\
\href{ronan.egan@dcu.ie}{ronan.egan@dcu.ie}
\medskip

\noindent
\textsc{Fintan Hegarty}\\
Mathematical Sciences Publishers, Berkeley, CA, USA\\
\href{fintan@msp.org}{fintan@msp.org}
\medskip

\noindent
\textsc{Padraig \'O Cath\'ain}\\
Worcester Polytechnic Institute, Worcester, MA, USA\\
\href{pocathain@wpi.edu}{pocathain@wpi.edu}

\bibliographystyle{abbrv}
\flushleft{
\bibliography{Biblio2020}
}

\end{document}